\documentclass[9pt, english, a4paper]{article}
\usepackage[latin1]{inputenc}
\usepackage[T1]{fontenc}
\usepackage{babel, graphicx, textcomp, varioref, amsfonts}
\usepackage{amsthm}
\usepackage{amssymb}
\usepackage{amsmath}
\usepackage{mathrsfs}
\usepackage[margin=1in]{geometry}

\newtheorem{theorem}{Theorem}[section]
\newtheorem{definition}[theorem]{Definition}
\newtheorem{lemma}[theorem]{Lemma}
\newtheorem{proposition}[theorem]{Proposition}
\newtheorem{corollary}[theorem]{Corollary}
\newtheorem{example}[theorem]{Example}
\newtheorem{remark}[theorem]{Remark}

\newcommand{\R}{\mathbb{R}} 
\newcommand{\X}{\mathbf{X}}
\newcommand{\Y}{\mathbf{Y}}
\newcommand{\B}{\mathbf{B}}
\newcommand{\LL}{\mathcal{L}}

\title{Rough linear PDE's with discontinuous coefficients - existence of solutions via regularization by fractional Brownian motion}

\author{Torstein Nilssen \thanks{Department of Mathematics, University of Oslo, Moltke Moes vei 35, P.O. Box 1053 Blindern, 0316 Oslo, Norway.
E-mail: torsteka@math.uio.no. 
Funded by  Norwegian Research Council (Project 230448/F20).} }

\begin{document}

\date{}

\maketitle

\begin{abstract}
We consider two related linear PDE's perturbed by a fractional Brownian motion. We allow the drift to be discontinuous, in which case the corresponding deterministic equation is ill-posed. However, the noise will be shown to have a regularizing effect on the equations in the sense that we can prove existence of solutions for almost all paths of the fractional Brownian motion.

\bigskip

MSC Classification Numbers: 60H15, 60H10, 60G22, 60H05, 60J55


\end{abstract}


\section{Introduction}

In this paper we study examples of the so called regularization by noise phenomenon for a class of linear equations perturbed by fractional Brownian motion. In short, this is the name given to the phenomenon that occurs when ill-posed deterministic equations becomes well-posed by adding stochastic terms.

More specifically, assume $b \in L^1(\R^d; L^{\infty}([0,T];\R^d)) \cap L^{\infty}( [0,T] \times \R^d ;\R^d )$ is a given function and let $B^H$ be a $d$-dimensional fractional Brownian motion (fBm). In this paper we will study two different but related linear stochastic PDE's.

The stochastic transport equation is the equation
\begin{align} \label{firstTransportEq}
\partial_t u(t,x) + b(t,x) \cdot \nabla u(t,x) + c(t,x) u(t,x) +  \nabla u(t,x) \cdot \dot{B}_t^H  = 0, \hspace{0.6cm} u(0,x)  = u_0(x)
\end{align}
where $u_0 \in C^1_b(\R)$, and we allow $c$ to be a distribution. In particular, we shall assume that $c$ is the distributional derivative of a bounded function. 

In the case that $c = div(b)$ this is called the continuity equation which we also may define as the measure valued equation
\begin{align} \label{firstContinuityEq}
\partial_t \mu_t + \textrm{div}(b \mu_t ) + \textrm{div} (\mu_t  \dot{B}_t^H) & = 0, \hspace{0.6cm} \mu|_{t=0}  = \mu_0
\end{align}
where $\mu_0 $ is a given measure. We see that $u(t,x)$ is equal to the Radon-Nykodim derivative of $\mu_t$ w.r.t. Lebesgue measure.

Both equations are related to the stochastic ordinary equation

\begin{equation} \label{flowSDE}
\phi_t(x) = x + \int_0^t b(r,\phi_r(x)) dr + B_t^H ,
\end{equation}
in the sense that the push-forward $\mu_t := (\phi_t)_{\sharp} \mu_0$ solves the continuity equation \eqref{firstContinuityEq} and the composition $u(t,x) := u_0(\phi_t^{-1}(x)) \exp\{ - \int_0^t c(r,\phi_r( \phi_t^{-1}(x))dr  \} $ solves the transport equation \eqref{firstTransportEq}.

This means that if we can show the regularization effect of the fBm on (\ref{flowSDE}) we can solve the corresponding stochastic PDE's as indicated above.

Both equations involves terms on the form $Y_t \dot{B}_t^H$, but we know that the fBm is $P$-a.s. not differentiable so one should integrate the equations in time to produce terms on the form $\int_0^t Y_s dB_s^H$. But even at this stage there is ambiguity. Indeed, since for $H \neq \frac{1}{2}$ the fBm is not a semi-martingale there is no It\^{o}-theory to make sense of this integral. Moreover, to enjoy the regularization effect of fBm on (\ref{flowSDE}) we need to have $H < \frac{1}{2}$. Since the trajectories of $B^H$ are $P$-a.s. H\"{o}lder continuous with exponent strictly smaller than $H$, and the solutions themselves cannot be expected to have higher regularity, also the integration theory by Young is out of reach for these equations.

As the title of the paper suggest, we shall interpret the integrals in the Rough Path setting, meaning we will use the iterated integrals of $B^H$ and the theory of controlled paths to give meaning to these integrals.

We will discuss the equations separately. For notational simplicity we write $B$ for the fBm.

\subsection{The stochastic continuity equation}

Integrating the continuity equation in time, and assuming we have the above mentioned integration theory, we get
\begin{equation} \label{CEWeakIntegralForm}
\mu_t + \int_0^t \textrm{div}(b \mu_s) ds + \int_0^t \textrm{div}( \mu_s dB_s) = \mu_0
\end{equation}
regarded as a measure valued equation, namely for every $\eta \in C^{\infty}_c(\R^d)$
$$
\mu_t ( \eta ) = \mu_0(\eta) + \int_0^t \mu_s ( b(s,\cdot) \cdot  \nabla \eta  )ds  +  \int_0^t  \mu_s ( \nabla \eta \cdot  dB_s )
$$
where $\mu_t(\eta) := \int_{\R^d} \eta(x) d\mu_t(x)$, $\mu_t ( \nabla \eta) = (\mu_t ( \partial_{x_1} \eta), \dots, \mu_t(\partial_{x_d} \eta))$ and $\cdot $ is the dot-product on $\R^d$.

We will show that the solution is on the form $\mu_t = (\phi_t)_{\sharp}\mu_0$. To see this, heuristically, take $\eta \in C^{\infty}_c(\R^d)$ and suppose we have some kind of It\^{o}-Stratonovich-formula for the fractional Brownian motion in the rough path setting. We should have
$$
\eta(\phi_t(x)) = \eta(x) + \int_0^t  \nabla \eta(\phi_r(x)) \cdot  b(r,\phi_r(x))  dr + \int_0^t  \nabla \eta(\phi_r(x)) \cdot  dB^H_r  .
$$
Integrating w.r.t. $\mu_0$ produces the desired formula provided we can use integration by parts.

The authors in \cite{BNP} show existence of a unique solution to (\ref{flowSDE}) and the results will be included in Section \ref{fBmSDEChapter}.

\subsection{The stochastic transport equation}

Integrating the linear transport equation in time gives
\begin{equation} \label{TEWeakIntegralForm}
u(t,x) + \int_0^t  b(s,x) \cdot  \nabla u(s,x) ds + \int_0^t c(s,x) u(s,x)ds + \int_0^t  \nabla u(s,x) \cdot  dB_s   = u_0(x) .
\end{equation}
It is well known that the corresponding deterministic equation might develop discontinuities when $b$ is irregular. Moreover, a weak formulation of the deterministic equation is not straightforward. Integrating against a test function $\eta \in C^{\infty}_c(\R)$, we see that the term $\int_{\R} b(t,x) \cdot \nabla u(t,x) \eta(x) dx$ does not allow for integration by parts unless there is some regularity on $b$. We will choose the noise in such a way that the solution is weakly differentiable, thus circumventing integration by parts. Notice however, that we will still use a (spatially) weak formulation of the equation.

The linear transport equation has been studied extensively. When the noise term is removed, Di Perna and Lions \cite{DiPernaLions}, showed that when $b \in L^1([0,T] ; W^{1,1}_{loc}(\R^d))$ with linear growth and $divb \in L^1([0,T] \times \R^d)$, a weak solution exists. Notice that the regularity restrictions on $b$ is needed in order to make a definition of a solution as indicated above.

The stochastic version driven by Brownian motion with Stratonovich formulation, i.e. $\int_0^t \nabla u(s,x) \circ dB_s$, has also received some attention. 
We mention the results in \cite{FedrizziFlandoli} and \cite{MNP}, developed simultaneously and independently using two somewhat different techniques.

An approach of using rough paths for regularization by noise was used in \cite{Catellier}, building on \cite{CatellierGubinelli}. The techniques of \cite{Catellier} and \cite{CatellierGubinelli} are similar in spirit to the present paper in the sense that they both use calculation on the occupation measures. The main advantage of \cite{Catellier} and \cite{CatellierGubinelli} is that they offer a more defined separation between the probabilistic considerations and the analysis of the involved ODE and PDE's, thus making the approach suitable for different types of driving noise. In the present paper one needs to carefully keep track of $P$-null sets because many of the estimates are only shown to be true under expectation. On the other hand it gives some flexibility since some of the expressions are semi explicit via the local time.

The paper \cite{Catellier} consider drifts for which $divb \in L^{\infty}([0,T] \times \R^d)$, and allow for linear growth. When $d=1$ this is restricts to (locally) Lipschitz drift, but when $d > 1$ this condition is much weaker than Lipschitz. Another difference from the current paper is that \cite{Catellier} considers $H \in (\frac{1}{3}, \frac{1}{2})$. For the technique in the current paper to work, we \emph{need} to have $H < \frac{1}{3}$ which makes the rough path theory a bit more involved.

The main advantage of the technique of the present paper is that the solution can easily be seen to be smoother in space, so that there is no need for integration by parts on the drift term, which is the reason for restricting to bounded divergence on $b$ in \cite{Catellier}.

In addition, we include a part where $d=1$ where the proof is much simpler. The proof is based on a local-time technique that was introduced in \cite{Nilssen} to study the Stochastic Heat Equation.


\subsection{Notation}

For Banach spaces $V,W$ we denote $\mathcal{L}(V;W)$ the set of all continuous linear mappings from $V$ to $W$. For simplicity we denote $\mathcal{L}(V) := \mathcal{L}(V;\R)$. If the spaces $V$ and $W$ are finite dimensional, and we can identify $\mathcal{L}(V \otimes W) $ with $\mathcal{L}(V;\LL(W))$.
In particular, for a sufficiently smooth function $f:V \rightarrow W$ the $k$'th derivative is considered as a map $\nabla ^kf : V \rightarrow \LL(V^{\otimes k};W)$.


For $T > 0$ define the simplex $\Delta^{(n)}(s,t) := \{ (r_1,\dots,r_n) \in [0,T]^n : s < r_1 <\dots < r_n < t\}$. For $\gamma > 0 $ denote by $C_2^{\gamma}([0,T];V)$ the space of all functions $f: \Delta^{(2)}(0,T) \rightarrow V$ such that $ \|f\|_{\gamma} := \sup_{s < t} \frac{|f(s,t)|}{|t-s|^{\gamma}} < \infty$. 
Given a function $X : [0,T] \rightarrow V$ its increment is denoted $X_{st} := X_t - X_s$.

For an integer $p$ the $p$-step truncated tensor algebra 
$$
T^{(p)}(\R^d) := \bigoplus_{n=0}^p (\R^d)^{\otimes n}
$$
is equipped with the product $(\mathbf{a} \otimes \mathbf{b})^{(n)} = \sum_{k=0}^n a^{(n-k)} \otimes b^{(k)}.$

We recall the following Taylor formula for a function $f : V \rightarrow W$ that is $m+1$ times differentiable
\begin{equation} \label{taylor3}
f(x) - f(y) = \sum_{k=1}^m \frac{\nabla^kf(y)}{k!} (x-y)^{\otimes k}  +  R^{f}_m(x,y)
\end{equation}
where $|R^f_m(x,y)| \lesssim |x-y|^{m+1}$. More specifically, we shall use the explicit formula
\begin{equation} \label{explicitTaylorRemainder}
R^f_m(x,y) = \frac{1}{m!} \int_0^1 \nabla^{m+1}f(y + u(x-y))(1-u)^m du (x-y)^{\otimes (m+1)}.
\end{equation}

\bigskip

We shall frequently use the space $L^1(\R^d; L^{\infty}([0,T]; \R^d))$ with norm denoted by
$$
\|f \|_{\infty, 1} := \int_{\R^d} \|f( \cdot ,x) \|_{L^{\infty}([0,T];\R^d)} dx. 
$$
For simplicity the norm in the space $L^{\infty}([0,T] \times \R^d; \R^d))$ will be denoted $\| \cdot \|_{\infty}$.



\section{Elements of Controlled Rough Paths} \label{ElementsOfControlledRoughPaths}

The theory of rough paths was first introduced by Terry Lyons in the late 90's, see \cite{Lyons98}. The insight of this work is that even though solutions to ODE's driven by rough signals are typically not continuous as a function of the signals themselves, adding extra information, namely the iterated integrals of the driving signals, one obtains a topology for which there is continuity of the solutions. The theory was further developed by Gubinelli, \cite{Gubinelli04} and \cite{Gubinelli10}, who introduced the notion of controlled paths which defines spaces that are well suited for constructing solutions of the rough ODE's.
In the present paper we shall use controlled paths as one of our main tools. See \cite{FrizHairer} for an introduction.

Throughout this section we fix some $\gamma \in (0, \frac{1}{2})$ and let $p$ be the integer part of $\frac{1}{\gamma}$. A $\gamma$-rough path is a mapping
\begin{align*}
\X: \Delta^{(2)}(0,T) & \rightarrow T^{(p)}(\R^d)  \\ 
(s,t) & \mapsto  (1, X_{st}^{(1)}, \dots , X_{st}^{(p)}) 
\end{align*}
that satisfies an algebraic (Chen's) relation
\begin{equation} \label{Chen}
\X_{st} = \X_{su} \otimes \X_{ut} ,
\end{equation}
and an analytic relation
\begin{equation} \label{HolderRelation}
|X_{st}^{(n)}| \lesssim |t-s|^{n\gamma} \hspace{1cm} n=1, \dots, p .
\end{equation}

We denote by $\mathscr{C}^{\gamma}$ the set of all rough paths equipped with the metric
$$
\varrho_{\gamma}(\X,\tilde{\X}) := \sum_{n=1}^p \sup_{t \neq s} \frac{|X_{st}^{(n)} -\tilde{X}^{(n)}_{st}|}{|t-s|^{n\gamma}} .
$$

Given a function $X \in C^1([0,T];\R^d)$ we can consider its canonical lift to a rough path
\begin{equation} \label{smoothLift}
\X_{st} := (1, X_{st}, \int_s^t X_{sr} \otimes \dot{X}_r dr , \dots, \int_{\Delta^{(p)}(s,t)} \dot{X}_{r_1} \otimes  \dots \otimes \dot{X}_{r_p} dr_1 \dots dr_p ).
\end{equation}

We denote by $\mathscr{C}^{\gamma}_g$ the closure of the canonical lift of $ C^1([0,T];\R^d)$ in the rough path topology
\footnote{sometimes written $\mathscr{C}_g^{0, \gamma}$ in the literature, whereas $\mathscr{C}^{\gamma}_g$ is reserved for paths satisfying (\ref{symmetryCondition}). While $\mathscr{C}_g^{0, \gamma}$ is strictly included in $\mathscr{C}_g^{ \gamma}$ one can use ``geodesic approximations'' and interpolation to show $\mathscr{C}_g^{ \gamma '} \subset \mathscr{C}_g^{0, \gamma} \subset \mathscr{C}_g^{\gamma} $ for $\gamma ' < \gamma$, so that one can still approximate elements satisfying (\ref{symmetryCondition}) at the expense of choosing a smaller $\gamma$.} .
An element $\X \in \mathscr{C}^{\gamma}_g$ will be referred to as a geometric rough path and it satisfies the identity
\begin{equation} \label{symmetryCondition}
\textrm{sym}(X_{st}^{(n)}) = \frac{1}{n!} \left( X_{st}^{(1)} \right)^{\otimes n}.
\end{equation}

Given a rough path $\X \in \mathscr{C}^{\gamma}$, we shall say that a mapping
\begin{align*}
\Y : [0,T]  & \longrightarrow \bigoplus_{n=1}^p \mathcal{L}( (\R^d)^{\otimes n}) \\
t & \longmapsto (Y^{(1)}_t, \dots Y^{(p)}_t) 
\end{align*}
is a controlled (by $\X$) path if the functions
$$
Y^{(k) \sharp}_{st} := Y^{(k)}_t - \sum_{n=k}^{p} Y_s^{(n)} X^{(n-k)}_{st} \hspace{0.4cm} k=1, \dots, p
$$
are such that $Y^{(k) \sharp} \in C^{(p+1-k)\gamma}_2([0,T]; \LL((\R^d)^{\otimes k})$, i.e.
\begin{equation} \label{controlledHolder}
|Y^{(k) \sharp}_{st}| \lesssim |t-s|^{(p+1-k) \gamma} .
\end{equation}

We denote by $\mathscr{D}^{p\gamma}_{\X}$ the set of all paths controlled by $\X$, and we equip this linear space with the semi-norm
$$
\|\Y\|_{\X} = \sum_{k=1}^{p} \|Y^{(k) \sharp}\|_{(p+1-k)\gamma} .
$$
Conditioned on $(Y^{(1)}_0, \dots, Y^{(p)}_0)$ we get a norm which controls the $\|\cdot\|_{\infty}$-norm of $\Y$ in the following way. We have  $Y_{t}^{(k) } = Y_{0t}^{(k) \sharp} + \sum_{n=k}^{p} Y_0^{(n)} X_{0t}^{(n-k)}$ so that 
\begin{align*}
\|Y^{(k)}\|_{\infty} & \leq T^{(p+1-k)\gamma} \|Y^{(k) \sharp} \|_{(p+1-k) \gamma} + \sum_{n=k}^{p} |Y_0^{(n)}| \|X\|^{(n-k)\gamma}T^{(n-k)\gamma} \\
 & \lesssim  \| \Y\|_{\X} + \varrho_{\gamma}(0,\X) |\Y_0| .
\end{align*}

If we consider two paths $\Y$ and $\tilde{\Y}$, controlled by  $\X$ and $\tilde{\X}$ respectively, we introduce the ``distance''
$$
\| \Y; \tilde{\Y} \|_{\X, \tilde{\X}} := \sum_{k=1}^{p} \| Y^{(k) \sharp} - \tilde{Y}^{(k) \sharp}\|_{(p+1-k) \gamma} .
$$
Similar as above have the following estimate
\begin{align} \label{supnormVSgammanorm}
\max_{n = 1, \dots , p} \| Y^{(n)} - \tilde{Y}^{(n)} \|_{\infty} & \leq \| \Y ; \tilde{\Y}\|_{\X, \tilde{\X}} + \varrho_{\gamma}(\X,0) |\Y_0 - \tilde{\Y}_0|  + \varrho_{\gamma}(\X, \tilde{\X}) |\tilde{\Y}_0| \notag .
\end{align}

We define the total space
$$ 
\mathscr{C}^{\gamma} \ltimes \mathscr{D}^{p \gamma} := \bigsqcup_{\X \in \mathscr{C}^{\gamma}} \{ \X \} \times \mathscr{D}^{p \gamma}_{\X}
$$
equipped with its natural topology, i.e. the weakest topology such that
\begin{align*}
\mathscr{C}^{\gamma} \ltimes \mathscr{D}^{p \gamma} & \longrightarrow \mathscr{C}^{\gamma} \times \bigoplus_{k=1}^{p} C_2^{(p+1-k) \gamma} ([0,T]; \LL((\R^d)^{\otimes k}))\\
(\X,\Y) & \longmapsto \left(\X, \oplus_{k=1}^{p} Y^{(k) \sharp} \right)
\end{align*}
is continuous.

If $f$ is a scalar valued function with higher H\"{o}lder regularity, i.e. $|f_{st}| \lesssim |t-s|^{\beta}$ for some $\beta \geq p\gamma$ and we take a controlled path $\Y \in \mathscr{D}_{\X}^{p\gamma}$ we can define a new controlled path $f\Y$.

\begin{lemma} \label{scalarMultiplication}
The mapping 
\begin{align*}
C^{\beta} \times \mathscr{D}_{\X}^{p\gamma} & \rightarrow \mathscr{D}_{\X}^{p\gamma} \\
(f,\Y) & \mapsto (fY^{(1)}, \dots, fY^{(p)})
\end{align*}
is bilinear and continuous when $\beta \geq p\gamma$.

\end{lemma}

\begin{proof}
To see that the mapping is well defined is sufficies to notice that 
$$
(fY)^{(k) \sharp}_{st} = f_{st} Y_t^{(k)} + f_s Y^{(k) \sharp}_{st}
$$
satisfies the required time-regularity when $ \beta \geq p \gamma$. To see continuity of this map we can similarly write
\begin{align*}
|(fY)^{(k) \sharp}_{st} - (\tilde{f} \tilde{Y})^{(k) \sharp}_{st}| &  \leq |t-s|^{\beta} \|f - \tilde{f}\|_{\beta}  \|Y^{(k)}\|_{\infty} \\
 & + |t-s|^{(p+1 - k)\gamma} \|f - \tilde{f} \|_{\infty} \|Y^{(k) \sharp} - \tilde{Y}^{(k) \sharp} \|_{(p+1-k)\gamma} .
\end{align*}
\end{proof}

\subsection{Integration of Controlled Rough Paths}

Following \cite{FrizHairer} we denote by $C_2^{\alpha, \beta}([0,T])$ the space of functions $\Xi : \Delta^{(2)}(0,T) \rightarrow \R$ such that 
$$
\| \Xi \|_{\alpha} := \sup_{s <t } \frac{|\Xi_{st}|}{|t-s|^{\alpha}} < \infty \hspace{.2cm} \textrm{and} \hspace{.2cm}  \|\delta \Xi\|_{\beta} := \sup_{s < u < t} \frac{| \delta \Xi_{sut}|}{|t-s|^{\beta}} < \infty
$$
where $\delta \Xi_{sut} := \Xi_{st} - \Xi_{su} - \Xi_{ut}$. We equip the space with the semi-norm $\|\Xi\|_{\alpha, \beta} := \|\Xi\|_{\alpha} + \|\delta \Xi\|_{\beta}$. The following result is sometimes referred to as the ``sewing lemma'':
\begin{lemma} \label{sewingLemma}
Assume $0< \alpha \leq 1 < \beta$. Then there exists a unique continuous linear map 
$$
\mathcal{I} : C^{\alpha, \beta}_2([0,T]) \rightarrow C^{\alpha}([0,T])
$$
such that $(\mathcal{I}\Xi)_0 = 0$ and 
$$
|(\mathcal{I}\Xi)_{st} - \Xi_{st}| \lesssim |t-s|^{\beta} .
$$

More specifically,
\begin{equation} \label{Idefinition}
\mathcal{I}(\Xi)_{st} = \lim_{|\mathcal{P}| \rightarrow 0} \sum_{[u,v] \in \mathcal{P}} \Xi_{uv}
\end{equation}
where $\mathcal{P}$ denotes a partition of $[s,t]$ and $|\mathcal{P}|$ its mesh. The limit can be taken along any sequence of partitions and is independent of this choice.
\end{lemma}

For a proof, see \cite{FrizHairer}. It is clear from (\ref{Idefinition})  that $C^{\theta}_2([0,T]) \subset ker(\mathcal{I})$ for $\theta > 1$.

We are ready to define the integral of a controlled rough path. For $\X \in \mathscr{C}^{\gamma}$ and $\Y \in \mathscr{D}_{\X}^{p \gamma}$ let 
$$
\Xi_{st} := \sum_{n=1}^p Y^{(n)}_s X^{(n)}_{st}  .
$$
Chen's relation (\ref{Chen}) gives $X_{st}^{(n)} = \sum_{k=0}^{n} X_{su}^{(n-k)} \otimes  X_{ut}^{(k)}$, so that 
\begin{align*}
\delta \Xi_{sut} & = \sum_{n=1}^p Y_s^{(n)} ( X_{st}^{(n)} - X_{su}^{(n)}) - \sum_{n=1}^p Y_u^{(n)}  X_{ut}^{(n)}  = \sum_{n=1}^p Y_s^{(n)} \sum_{k=1}^n X_{su}^{(n-k)} \otimes  X_{ut}^{(k)} - \sum_{n=1}^p Y_u^{(n)}  X_{ut}^{(n)} \\
& = \sum_{k=1}^p \sum_{n=k}^p Y_s^{(n)}  X_{su}^{(n-k)} \otimes X_{ut}^{(k)} - \sum_{k=1}^p Y_u^{(k)}  X_{ut}^{(k)}  = \sum_{k=1}^p \left( \sum_{n=k}^p Y_s^{(n)}  X_{su}^{(n-k)}   - Y_u^{(k)}  \right) X_{ut}^{(k)} \\
& = - \sum_{k=1}^p Y_{su}^{(k) \sharp} X_{ut}^{(k)} .
\end{align*}
From (\ref{HolderRelation}) and (\ref{controlledHolder}) each term can be bounded by $C|t-s|^{(p+1) \gamma}$ for an appropriate constant $C$. Consequently $| \delta \Xi_{sut} | \lesssim |t-s|^{(p+1) \gamma}$. Since $(p+1) \gamma > 1$  we arrive at the following definition:

\begin{definition}
Let $\X \in \mathscr{C}^{\gamma}$ and let $\Y \in  \mathscr{D}^{p\gamma}_{\X}$. We define the rough path integral of $\Y$ w.r.t. $\X$ as 
\begin{equation} \label{roughIntegralDef}
\int_s^t \Y_r d\X_r := (\mathcal{I}\Xi)_{st}
\end{equation}
with $\mathcal{I}$ and $\Xi$ as above.

\end{definition}

\begin{remark} \label{smoothIntegral}
For a smooth path $X$ with its geometric lift (\ref{smoothLift}) the rough path integral and the usual calculus coincide, i.e.
$$
\int_s^t Y_r d \X_r = \int_s^t Y_r \dot{X}_r dr,
$$
for all $Y \in C^{\gamma}([0,T]; \LL(\R^d))$. Indeed, we may define $Y^{(n)} = 0$ for $n=2, \dots p$. Even though in general (\ref{controlledHolder}) is not satisfied for $k=1$, if we define 
$$
\Xi_{st} := Y_s X_{st}
$$
we get $\delta \Xi_{sut} = -Y_{su}X_{ut}$ so that $\Xi \in C^{ 1,1 + \gamma}_2([0,T])$. 

\end{remark}

The rest of this section is devoted to obtaining a ``local Lipschitz''-type estimate when we regard the above as a mapping
$$
\mathscr{C}^{\gamma} \ltimes \mathscr{D}^{p\gamma} \rightarrow C^{\gamma, (p+1)\gamma}_{2}([0,T]).
$$
Indeed, let $\X, \tilde{\X} \in \mathscr{C}^{\gamma}$ and let $\Y$ and $\tilde{\Y}$ be controlled by $\X$ and $\tilde{\X}$ respectively. Define $\Xi$ as before and 
$$
\tilde{\Xi}_{st} := \sum_{n=1}^p \tilde{Y}_s^{(n)}\tilde{X}_{st}^{(n)}.
$$
\begin{lemma} \label{continuousIntegration}
Assume $\varrho_{\gamma}(0,\X), \|\Y\|_{\X},|\Y_0| \leq M$ for some constant $M$, and similarly for $\tilde{\X}$ and $\tilde{\Y}$. Then there exists a constant $C_M$ such that 
$$
\| \Xi - \tilde{\Xi} \|_{\gamma, (p+1)\gamma} \leq C_M( |\Y_0 - \tilde{\Y}_0 | + \|\Y ; \tilde{\Y} \|_{\X; \tilde{\X}} + \varrho_{\gamma}(\X, \tilde{\X})) .
$$
\end{lemma}

\begin{proof}
We begin by decomposing
\begin{align*}
\Xi_{st} - \tilde{\Xi}_{st} & =  \sum_{n=1}^p Y_s^{(n)} X_{st}^{(n)} - \sum_{n=1}^p \tilde{Y}_s^{(n)}\tilde{X}_{st}^{(n)} \\
& = \sum_{n=1}^p Y_s^{(n)} (X_{st}^{(n)} - \tilde{X}_{st}^{(n)}) + \sum_{n=1}^p (Y_s^{(n)} - \tilde{Y}_s^{(n)})\tilde{X}_{st}^{(n)} 
\end{align*}
so that 
\begin{align*}
|\Xi_{st} - \tilde{\Xi}_{st}| & \leq   \sum_{n=1}^p \|Y^{(n)}\|_{\infty} \|X^{(n)} - \tilde{X}^{(n)}\|_{n\gamma}|t-s|^{n\gamma} \\
&  + \sum_{n=1}^p \|Y^{(n)} - \tilde{Y}^{(n)}\|_{\infty} \| \tilde{X}^{(n)}\|_{n \gamma} |t-s|^{n\gamma} \\
& \leq   |t-s|^{\gamma} \max_{n=1, \dots, p} \|Y^{(n)}\|_{\infty}  \varrho_{\gamma}(\X, \tilde{\X})   \\
&  + |t-s|^{\gamma} \varrho_{\gamma}(0, \tilde{\X}) \max_{n=1, \dots, p} \|Y^{(n)} - \tilde{Y}^{(n)}\|_{\infty}   .
\end{align*}

Using (\ref{supnormVSgammanorm}) we can find a constant $\tilde{C}_M$ such that  
\begin{align*}
\|\Xi - \tilde{\Xi} \|_{\gamma} &  \leq \tilde{C}_M ( \| \Y ; \tilde{\Y}\|_{\X, \tilde{\X}} + |\Y_0 - \tilde{\Y}_0| \\
& + \varrho_{\gamma}(\X, \tilde{\X}) ) .
\end{align*}
Similarly,
\begin{align*}
\delta \Xi_{sut} - \delta \tilde{\Xi}_{sut} & = -\sum_{n=1}^p Y_{su}^{(n) \sharp } X_{ut}^{(n)} + \sum_{n=1}^p \tilde{Y}_{su}^{(n) \sharp}\tilde{X}_{ut}^{(n)} \\
& = -\sum_{n=1}^p Y_{su}^{(n) \sharp} (X_{ut}^{(n)} - \tilde{X}_{ut}^{(n)}) + \sum_{n=1}^p (Y_{su}^{(n) \sharp} - \tilde{Y}_{su}^{(n)\sharp })\tilde{X}_{ut}^{(n)} 
\end{align*}
so that 
\begin{align*}
\| \delta (\Xi - \tilde{\Xi}) \|_{(p+1)\gamma}  & \leq  \sum_{n=1}^p \|Y^{(n) \sharp }\|_{(p+1-n)\gamma} \| X^{(n)} - \tilde{X}^{(n)}\|_{n\gamma} \\
& + \sum_{n=1}^p \| Y^{(n) \sharp} - \tilde{Y}^{(n)\sharp }\|_{(p+1-n)\gamma}  \|\tilde{X}^{(n)} \|_{n\gamma} \\
& \leq  M ( \varrho_{\gamma}(\X, \tilde{\X})  + \|\Y; \tilde{\Y}\|_{\X, \tilde{\X}} )  . 
\end{align*}
\end{proof}

\subsection{Controlling solutions of ODE's} \label{controllingODEs}
In this section we will show how to control solutions to ODE's perturbed by a rough path $\X \in \mathscr{C}^{\gamma}$. Fix a function $b \in C^1_b([0,T] \times \R^d;\R^d)$ and denote by $\phi_{\cdot}(x)$ the solution of the perturbed ODE 
\begin{equation} \label{roughODE}
\phi_t(x) = x + \int_0^t b(r,\phi_r(x)) dr + X_t .
\end{equation}


When there is no chance of confusion we shall denote the solution to (\ref{roughODE}) by $\phi_t$ for notational convenience. Notice that we shall later on be interested in $\phi_t$ as a function of $x$, but for this section we leave it fixed.

We have
\begin{align*}
\phi_{st}  = \int_s^t b(r,\phi_r)dr + X_{st}  =: R_{st}^{\phi} + X_{st} 
\end{align*}
where $|R_{st}^{\phi}| \lesssim |t-s|$ by the boundedness of $b$. Let $f \in C^{p}_b(\R^d;\R^d)$, so that we can view $\nabla^k f : \R^d \rightarrow \LL((\R^d)^{\otimes (k+1)})$. We shall lift the composition $f (\phi)$ to a controlled path in $\mathscr{D}^{p\gamma}_{\X}$.

\begin{lemma} \label{Lemma:FunctionComposition}
  Assume $\X$ is a geometric rough path. Then the mapping $s \mapsto (f(\phi_s), \dots, \nabla^{p-1} f(\phi_s))$ belongs to $\mathscr{D}^{p \gamma}_{\X}$, i.e. if we introduce the ad-hoc notation
$$
f(\phi)^{(k) \sharp}_{st}  := \nabla^{k-1} f(\phi_t) - \sum_{n=k}^{p} \nabla^{n} f(\phi_s) X_{st}^{(n-k)} \hspace{.5cm} k=1, \dots, p
$$
we have $f(\phi)^{(k) \sharp} \in C^{(p+1-k)\gamma}_2([0,T]; \LL((\R^d)^{\otimes (k+1)}))$.

\end{lemma}

\begin{proof}
  Begin by writing
  $$
  \phi_{st}^{\otimes n} = (R_{st}^{\phi} + X_{st})^{\otimes n} = \sum_{q=0}^n \binom{n}{q} \textrm{sym}( (R_{st}^{\phi})^{\otimes (n-q)} \otimes X_{st}^{\otimes q}) .
  $$ 
For a sufficiently smooth function $g : \R^d \rightarrow \LL( V)$ where $V$ is a finite-dimensional Banach space, we have from Taylor's formula 
\begin{align} \label{controlledRemainder}
\notag g(\phi_t) - g(\phi_s) & = \sum_{n=1}^m \frac{\nabla^n g (\phi_s)}{n!} (\phi_{st})^{\otimes n} + R^g_m(\phi_s, \phi_t) \\
& = \sum_{n=1}^m \nabla^n g(\phi_s) X_{st}^{(n)} + R^g_m(\phi_s, \phi_t)  \\
\notag & \hspace{.2cm} +  \sum_{n=1}^m \sum_{q=1}^n\binom{n}{q}  \frac{\nabla^ng(\phi_s)}{n!}    ( (R_{st}^{\phi})^{\otimes (n-q)} \otimes X_{st}^{\otimes q})  .
\end{align}
In the above we have used that $\X$ satisfies (\ref{symmetryCondition}) so that $\nabla^n g(\phi_s)\frac{X_{st}^{\otimes n}}{n!}=   \nabla^n g(\phi_s) X_{st}^{(n)}$ since $\nabla^ng$ only acts on symmetric tensors. Furthermore, the second term $ \lesssim |\phi_{st}|^{m+1} \lesssim |t-s|^{(m+1)\gamma}$, and the third term $ \lesssim |t-s|$. With $g = \nabla^{k} f$ and $m = p-k -1$ it follows that $f(\phi)^{(k) \sharp} \in C^{(p-k)\gamma}_2([0,T]; \LL((\R)^d)^{\otimes (k+1)})$, thus proving the lemma.
\end{proof}

\begin{corollary} \label{roughIntegralDefCor}
For $f \in C^{p}_b(\R^d;\R^d)$ we may define $\int f(\phi_r) d\X_r$ as the rough path integral of $f(\phi)$ w.r.t. $\X$ as in (\ref{roughIntegralDef}).
\end{corollary}

\subsection{Stability w.r.t. the driving path} \label{stabilityODEs}

The purpose of this section is to prove local Lipschitz continuity of the mapping 
\begin{align*}
\mathscr{C}^{\gamma} & \rightarrow \mathscr{C}^{\gamma} \ltimes \mathscr{D}^{p \gamma} \\
\X & \mapsto (\X, f (\phi)) 
\end{align*}
where $\phi$ is the solution to (\ref{roughODE}), $ f \in C^{p}_b(\R^d;\R^d)$ and $f (\phi)$ denotes the lift as described in the previous section. We begin with some trivial bounds, namely let $\tilde{\X} \in \mathscr{C}^{\gamma}$ and denote by $\tilde{\phi}$ the solution to (\ref{roughODE}) when we replace $\X$ by $\tilde{\X}$, i.e.
\begin{align*}
\tilde{\phi}_{st} & = \int_s^t b(r,\tilde{\phi}_r)dr + \tilde{X}_{st}  =: R^{\tilde{\phi}}_{st} + \tilde{X}_{st} .
\end{align*}

One can check that (see \cite{Catellier}, Lemma A.7) 
\begin{equation} \label{flowStability}
\| \phi - \tilde{\phi} \|_{\gamma} \leq C(T, \nabla b) \| X - \tilde{X} \|_{\gamma} .
\end{equation}
Clearly this implies $\| \phi - \tilde{\phi} \|_{\gamma} \lesssim \varrho_{\gamma}(\X, \tilde{\X})$ and also $\| R^{\phi} - R^{\tilde{\phi}}\|_{\gamma} \lesssim \varrho_{\gamma}(\X, \tilde{\X})$.

It follows that $ \| \phi^{\otimes n} - \tilde{\phi}^{\otimes n} \|_{n\gamma} \lesssim \varrho_{\gamma}(\X, \tilde{\X})$ by induction: assume this holds for $n-1$. Then
\begin{align*}
|\phi_{st}^{\otimes n} - \tilde{\phi}_{st}^{\otimes n}| & \leq |\phi_{st}^{ \otimes (n-1)}| |\phi_{st} - \tilde{\phi}_{st}| + |\phi_{st}^{\otimes (n-1)} - \tilde{\phi}_{st}^{\otimes (n-1)}| |\tilde{\phi}_{st}| \\
 & \leq 2 |t-s|^{n\gamma} \varrho_{\gamma}(\X, \tilde{\X})
\end{align*}
by the induction hypothesis combined with (\ref{flowStability}).


The main result of this section is the following.

\begin{lemma} \label{localLipschitz}
Assume $\varrho_{\gamma}(\X,0),\varrho_{\gamma}(\tilde{\X},0)  \leq M$ and $f \in C^{p}_b(\R^d;\R^d)$. Then there exists a constant $C_M$ such that
$$
\| f (\phi) ; f ( \tilde{\phi} )\|_{\X, \tilde{\X}} \leq C_M \varrho_{\gamma}(\X, \tilde{\X}) .
$$
\end{lemma}

\begin{proof}
We shall use the formula (\ref{controlledRemainder}) to show that $\| f(\phi)^{(k) \sharp}  - f(\tilde{\phi})^{(k) \sharp} \|_{(p-k) \gamma} \leq C_M \varrho_{\gamma}(\X, \tilde{\X})$, which will prove the claim. To this end for a function $g$ smooth enough, we have that the remainder term of the Taylor expansion satisfies
\begin{align*}
 R^g_m(\phi_s, \phi_t) - R^g_m(\tilde{\phi}_s,\tilde{\phi}_t ) & = \int_0^1 \frac{(1-r)^{m+1}}{ m!} \nabla^{m+1}g(\phi_s + r \phi_{st}) dr \left( \phi_{st}^{\otimes (m+1)} - \tilde{\phi}_{st}^{ \otimes (m+1)} \right)  \\
& +  \int_0^1 \frac{(1-r)^{m+1}}{ m!} \left( \nabla^{m+1} g (\phi_s + r \phi_{st})  - \nabla^{m+1}g (\tilde{\phi}_s + r \tilde{\phi}_{st}) \right)dr \left( \tilde{\phi}_{st}^{ \otimes (m+1)} \right) . 
\end{align*}
For the first term above we have $\lesssim |t-s|^{(m+1) \gamma} \|\nabla^{m+1}g\|_{\infty} \varrho_{\gamma}(\X, \tilde{\X})$. For the second term we use, uniformly in $r \in [0,1]$
\begin{align*}
| \nabla^{m+1}g(\phi_s + r \phi_{st})  - \nabla^{m+1}g(\tilde{\phi}_s + r \tilde{\phi}_{st})| & \leq \| \nabla^{m+2}g \|_{\infty} ( |\phi_s - \tilde{\phi}_s| + r |\phi_{st} - \tilde{\phi}_{st}|) \\
 & \lesssim \| \nabla^{m+2}g\|_{\infty} \varrho_{\gamma}(\X, \tilde{\X}) .
\end{align*}
Together with the bound $|\tilde{\phi}_{st}^{\otimes (m+1)}| \lesssim |t-s|^{(m+1) \gamma}$ we see that 
$$
\| R^g_m(\phi_{\cdot}, \phi_{\cdot}) - R^g_m(\tilde{\phi}_{\cdot},\tilde{\phi}_{\cdot} ) \|_{(m+1) \gamma} \lesssim \varrho_{\gamma}(\X, \tilde{\X}) .
$$

Fix integers $q \geq 1$ and $n \geq 0$. Using the estimate $|a \otimes b - a' \otimes b'| \leq |a-a'||b| + |a'||b-b'|$ repeatedly, it is easy to check that
$$
|\nabla g(\phi_s)(X_{st}^{\otimes n} \otimes (R_{st}^{\phi})^{\otimes q}) - \nabla g(\tilde{\phi}_s)(\tilde{X}_{st}^{\otimes n}\otimes  (R_{st}^{\tilde{\phi}})^{\otimes q}|  \lesssim |t-s| \varrho_{\gamma}(\X, \tilde{\X}) .
$$

This combined with (\ref{controlledRemainder}) gives 
$$
\| f (\phi)^{(k) \sharp} -  f (\tilde{\phi})^{(k) \sharp} \|_{(p-k)\gamma} \lesssim \varrho_{\gamma}(\X, \tilde{\X}) 
$$
which ends the proof of the lemma.
\end{proof}

Combining the above Lemma, Lemma \ref{continuousIntegration} and Remark \ref{smoothIntegral} we get

\begin{corollary} \label{smoothStability}
Let $\X \in \mathscr{C}^{\gamma}_g$. Then there exists a family of smooth paths $X^{\epsilon}$ such that 
$$
\int_0^{\cdot} f(\phi_r^{\epsilon}) \dot{X}_r^{\epsilon} dr \rightarrow \int_0^{\cdot} f(\phi_r) d\X_r \hspace{.5cm} \textrm{ in  } C^{\gamma} ([0,T]),
$$
as $\epsilon \rightarrow 0 $.

\end{corollary}

\subsection{Stability w.r.t. the drift}

Let us fix $\X \in \mathscr{C}^{\gamma}$ and we consider the ODE (\ref{roughODE}). Assume we have a sequence of functions $b_{\epsilon}$ such that there exists a solution of for every $\epsilon > 0$ to
$$
\phi^{\epsilon}_t = x + \int_0^t b_{\epsilon}(r,\phi^{\epsilon}_r)dr + X_t .
$$
We will show stability in the sense of controlled rough paths when we \emph{assume} that $\phi^{\epsilon}$ converges in an appropriate topology. This convergence will be shown to hold in Proposition \ref{HolderConvergence} for our particular case.

\begin{lemma} \label{roughIntegralConvergence}
Assume $\phi^{\epsilon}$ converges in $C^{\gamma}$ to the solution of (\ref{roughODE}). Then for any $f \in C^{p}_b(\R^d;\R^d)$ we have that the lift of $f (\phi^{\epsilon})$ converges in $\mathscr{D}^{p\gamma}_{\X}$ to $f(\phi)$, and as $\epsilon \rightarrow 0$
$$
\int_0^{\cdot} f(\phi^{\epsilon}_r) d\X_r \rightarrow \int_0^{\cdot} f(\phi_r)d\X_r
$$
where the above convergence is in $C^{\gamma}$.

\end{lemma}

\begin{proof}
Note that the second claim follows from the first in connection with Remark \ref{sewingLemma}.

To see the first claim, one has to show 
$$
\lim_{\epsilon \rightarrow 0} \| f( \phi)^{(k) \sharp} - f( \phi^{\epsilon})^{(k) \sharp} \|_{(p-k) \gamma} = 0
$$
for all $k = 0,1, \dots, p-1$. The proof follows the same lines as the proof of Lemma \ref{localLipschitz} with minor modifications, noting that $\X = \tilde{\X}$. 
\end{proof}

\subsection{An It\^{o}-Stratonovich formula}

For the sake of being self-contained, we include a change-of-variable formula for our particular case. Let $\eta \in C^{\infty}_c(\R^d)$ and assume $\phi_{\cdot}$ solves (\ref{roughODE}). If $X$ is a smooth path usual calculus yields,
$$
\frac{d}{dt} \eta(\phi_t) =  \nabla\eta(\phi_t)  \cdot b(t,\phi_t)   +   \nabla \eta(\phi_t) \cdot  \dot{X}_t  .
$$

We can generalize this to geometric rough paths.

\begin{lemma} \label{ItoFormula}
Suppose $\eta \in C^{\infty}_c(\R^d)$ and $\X$ is a rough path above $X$. Then we have
$$
\eta(\phi_t) = \eta(x) + \int_0^t  \eta(\phi_r) \cdot b(r,\phi_r)  dr  +  \int_0^t  \nabla \eta(\phi_r) d\X_r  .
$$
where the last term is the rough path integral.

\end{lemma}

\begin{proof}
  Let $ 0 \leq u \leq v \leq t$ and use Taylor's formula to write, as in (\ref{controlledRemainder})
 \begin{align*}
   \eta(\phi)_{uv} & = \sum_{n=1}^p \frac{\nabla^n\eta (\phi_u)}{n!} (\phi_{uv})^{\otimes n} + R^{\eta}_p(\phi_u,\phi_v)  =  \nabla\eta(\phi_u) R^{\phi}_{uv} + \sum_{n=1}^p \nabla^n \eta(\phi_u) X_{uv}^{(n)} + \Xi_{uv}
 \end{align*}

 where
$$
\Xi_{uv} := R^{\eta}_p(\phi_u,\phi_v) + \sum_{n=2}^p \sum_{q=1}^{n-1} \binom{n}{q} \frac{\nabla^n\eta(\phi_u)}{n!} ( (R_{uv}^{\phi})^{\otimes (n-q)} \otimes X_{uv}^{\otimes q})
$$
and notice that $\Xi \in C^{1+ \gamma}_2([0,T]) \subset ker(\mathcal{I})$. We have
\begin{align*}
  \lim_{|\mathcal{P}| \rightarrow 0} \sum_{[u,v] \in \mathcal{P}}  \nabla\eta(\phi_{u}) \cdot  \int_u^v b(r,\phi_r)dr  & = \lim_{|\mathcal{P}| \rightarrow 0} \int_0^t  \sum_{[u,v] \in \mathcal{P}}  \nabla\eta(\phi_{u})1_{[u,v]}(r) \cdot   b(r,\phi_r)  dr \\
  & = \int_0^t  \nabla\eta(\phi_r) \cdot  b(r,\phi_r)  dr
\end{align*}
where we used continuity of $\nabla\eta$ and dominated convergence in the last step to take in the limit. Note that the above reasoning does not use any regularity requirements on $b$.

Finally, we have
\begin{align*}
  \eta(\phi_t)  - \eta(x) & = \mathcal{I} (\eta(\phi_{\cdot}))_{0,t} = \lim_{|\mathcal{P}| \rightarrow 0} \sum_{[u,v] \in \mathcal{P}} \eta(\phi)_{uv} \\
  & = \lim_{|\mathcal{P}| \rightarrow 0} \sum_{[u,v] \in \mathcal{P}} \left( \nabla\eta(\phi_u) R^{\phi}_{uv} + \sum_{n=1}^p \nabla^n \eta(\phi_u) X_{uv}^{(n)} + \Xi_{uv} \right) \\
  & = \mathcal{I}( \nabla\eta(\phi_{\cdot}) R^{\phi}_{\cdot \cdot}) +  \mathcal{I}(\sum_{n=1}^p \nabla^n \eta(\phi_{\cdot}) X_{\cdot \cdot}^{(n)}) + \mathcal{I}(\Xi) \\
  & = \int_0^t  \nabla\eta(\phi_r)  \cdot b(r,\phi_r)  dr  +  \int_0^t  \nabla \eta(\phi_r) d\X_t 
  \end{align*}
by definition of the rough path integral.
\end{proof}

\subsection{Integrated ODE's}

To emphasize that the solution of (\ref{roughODE}) depends on the initial value $x$, we denote its solution by $\phi_{\cdot}(x)$, i.e.
$$
\phi_t(x) = x + \int_0^t b(r, \phi_r(x)) dr + X_t .
$$
Let $\nu$ be a finite signed measure on $\R^d$, and $f = (f^{(1)}, \dots, f^{(d)}) \in C^p_b(\R^d;\R^d)$.  In later chapters we shall be interested in expressions on the form
$$
\nu( f(\phi_{\cdot})) := \left(\int_{\R} f^{(1)} (\phi_{\cdot}(x)) d \nu(x), \dots, f^{(d)} (\phi_{\cdot}(x)) d \nu(x) \right) \in \LL(\R^d) 
$$
as a controlled path in order to define $ \int_0^t \nu( f(\phi_{r})) d \X_r$ in the rough path sense. Similar results as the previous chapters holds, summarized below. 

\begin{proposition} \label{summaryProposition}
Retain the hypotheses and notations respectively from Corollary \ref{roughIntegralDefCor}, Corollary \ref{smoothStability} and Lemma \ref{roughIntegralConvergence}. The following holds.

\end{proposition}

\begin{enumerate}

\item
The rough path integral $ \int_0^t \nu( f(\phi_{r})) d \X_r$ is well defined.

\item
Let $\X \in \mathscr{C}^{\gamma}_g$. Then there exists a family of smooth paths $X^{\epsilon}$ such that 
$$
\int_0^{\cdot} \nu( f(\phi^{\epsilon}_{r})) \dot{X}_r^{\epsilon} dr \rightarrow \int_0^{\cdot} \nu( f(\phi_{r})) d\X_r \hspace{.5cm} \textrm{ in  } C^{\gamma} ([0,T]),
$$
as $\epsilon \rightarrow 0 $, where $\phi^{\epsilon}$ denotes the solution of \eqref{roughODE} with $X$ replaced by $X^{\epsilon}$.

\item
If $   \nu( f(\phi^{\epsilon}_{\cdot})) \rightarrow \nu( f(\phi_{\cdot}))  $ in $C^{\gamma}$ we have
$$
\int_0^{\cdot}  \nu( f(\phi^{\epsilon}_{r}))  d\X_r \rightarrow \int_0^{\cdot}  \nu( f(\phi_{r}))  d\X_r \hspace{.5cm} \textrm{ in  } C^{\gamma} ([0,T]),
$$
as $\epsilon \rightarrow 0 $.

\end{enumerate}

\begin{proof}

Begin with the first assertion. Integrating (\ref{controlledRemainder}) w.r.t. $\nu$ gives
\begin{align*}
\int_{\R^d} f( \phi(x))_{st}^{(k) \sharp} d\nu(x)   & = \sum_{n=1}^{p-k-1} \sum_{q=1}^n \int_{\R^d} \frac{\nabla f^{(k+n)}(\phi_s(x))}{q!}  (R^{\phi(x)}_{st})^{\otimes q} \otimes X_{st}^{\otimes (n-q)}d\nu(x) \\
&  + \int_{\R^d}  R^{\nabla^kf}_{p-k-1}(\phi_s(x), \phi_t(x))  d\nu(x).
\end{align*}
Since $\nu$ is finite and $b$ is bounded we get for each $k,n$ and $q$
$$
|\int_{\R^d} \frac{\nabla^{k+n}f(\phi_s(x))}{q!}  (R^{\phi(x)}_{st})^{\otimes q} d\nu(x)| \lesssim |t-s|.
$$
Furthermore
$$
|\int_{\R^d}  R^{\nabla^{k}f}_{p-k-1}(\phi_s(x), \phi_t(x))  d\nu(x)| \lesssim \int_{\R^d} |\phi_{st}(x)|^{p-k} d\nu(x) \lesssim |t-s|^{(p-k)\gamma},
$$
so that $\int_{\R^d} f( \phi_{\cdot}(x)) d\nu(x) $ is a controlled path and 
$$
\left(\int_{\R^d} f( \phi(x)) d\nu(x)\right)_{st}^{(k) \sharp}  = \int_{\R^d} f( \phi(x))_{st}^{(k) \sharp} d\nu(x). 
$$

Using linearity, boundedness of $b$ and dominated convergence the reader is invited to complete the remaining steps of the proof.
\end{proof}

\section{Fractional Brownian motion and Girsanov's theorem}

Let $B = \{B_t, t\in [0,T]\}$ be a 1-dimensional \emph{fractional Brownian motion} (fBm) with Hurst parameter $H\in ( 0,\frac{1}{2})$, i.e. a centered Gaussian process with covariance
$$
R_H(t,s):= E[B_t  B_s]= \frac{1}{2}\left(t^{2H} + s^{2H} - |t-s|^{2H} \right) .
$$
Observe that $B$ has stationary increments and H\"{o}lder continuous trajectories of index $H-\varepsilon$ for all $\varepsilon>0$. 

Denote by $\mathcal{E}$ the set of step functions on $[0,T]$ and denote by $\mathcal{H}$ the Hilbert space defined as the closure of $\mathcal{E}$ with respect to the inner product 
$$
\langle 1_{[0,t]} , 1_{[0,s]}\rangle_{\mathcal{H}} = R_H(t,s).
$$
The mapping $1_{[0,t]} \mapsto B_t$ can be extended to an isometry between $\mathcal{H}$ and a Gaussian subspace of $L^2(\Omega)$. 

For a function $f \in L^2([a,b])$, we define the \emph{left fractional Riemann-Liouville integral} by
$$
I^{\alpha}_{0+} f(x) = \frac{1}{\Gamma(\alpha)} \int_0^x (x-y)^{\alpha - 1} f(y) dy
$$
for $\alpha > 0$. Denote by $I^{\alpha}_{0+}(L^2([a,b]))$ the image of $L^2([a,b])$ under $I^{\alpha}_{0+}$ and by $D^{\alpha}_{a+}$ its inverse.

We define $K_H(t,s)$ as
$$
K_H(t,s) = c_H \Gamma \left( H+\frac{1}{2}\right) s^{\frac{1}{2}-H} \left( D_{t^-}^{\frac{1}{2}-H} u^{H-\frac{1}{2}}\right)(s),
$$
for some constant $c_H$ and write $K_H$ for the operator from $L^2([0,T])$ onto $I_{0+}^{H+\frac{1}{2}}(L^2)$ associated with the kernel $K_H(t,s)$.
It follows that 
$$ R_H(t,s) = \int_0^{t\wedge s} K_H(t,u)K_H(s,u)du.$$

Moreover, if $W = \{ W_t : t \in [0,T]\}$ is a standard Brownian motion $B$ can be represented as
\begin{equation} \label{fBmRepresentation}
B_t = \int_0^t K_H(t,s) dW_s.
\end{equation}

A d-dimensional fractional Brownian motion is a d-dimensional process where the components are independent 1-dimensional fractional Brownian motions.

\begin{theorem}[Girsanov's theorem for fBm]\label{girsanov}
Let $u=\{u_t, t\in [0,T]\}$ be an $\R^d$-valued, $\{\mathcal{F}_t\}_{t\in[0,T]}$-adapted process with integrable trajectories  and set
$\widetilde{B}_t = B_t + \int_0^t u_s ds, \quad t\in [0,T].$
Assume that
\begin{itemize}
\item[(i)] $\int_0^{\cdot} u_s ds \in (I_{0+}^{H+\frac{1}{2}} (L^2 ([0,T]))^d$, $P$-a.s.

\item[(ii)] $E[\xi_T]=1$ where
$$\xi_T := \exp\left\{-\int_0^T  K_H^{-1}\left( \int_0^{\cdot} u_r dr\right)(s) \cdot dW_s  - \frac{1}{2} \int_0^T \left|K_H^{-1} \left( \int_0^{\cdot} u_r dr \right)(s) \right|^2ds \right\}.$$
\end{itemize}
Then the shifted process $\widetilde{B}$ is an $\{\mathcal{F}_t\}_{t\in[0,T]}$-fractional Brownian motion with Hurst parameter $H$ under the new probability $\widetilde{P}$ defined by $\frac{d\widetilde{P}}{dP}=\xi_T$.

Moreover, for every $p > 1$ we have $E[ |\xi_T|^p] \leq C_p(\|b\|_{\infty})$, where $C_p(\cdot)$ is an increasing function.
\end{theorem}

For a proof we refer to \cite{NualartOuknine}. In particular, the moment-estimate is found in the proof of Theorem 3, \cite{NualartOuknine}.

In the absence of the independent increments one has for $H= \frac{1}{2}$, we shall need the following fact (see \cite[Theorem 3.1]{Xiao}).

\begin{lemma}
The fractional Brownian motion is strong local non-deterministic, i.e. there exists a constant $c$ such that 
\begin{equation} \label{LocalNonDeterminism}
Var( B_t : (B_s)_{s : |t - s| \geq r} ) \geq c r^{2H} .
\end{equation}
\end{lemma}

Given an $m$-dimensional Gaussian vector $Z \sim \mathcal{N}(0,\Sigma)$ it is well known that 
\begin{equation} \label{GaussianCovariance}
|\Sigma| = Var(Z_m) Var(Z_{m-1} | Z_m) \dots Var(Z_1 | Z_m \dots Z_2),
\end{equation}
and so from Cramer's rule we get
\begin{equation} \label{CramersRule}
(\Sigma^{-1})_{j,j} = (Var(Z_j | Z_1, \dots, \widehat{Z_j}, \dots, Z_m))^{-1} 
\end{equation}

We shall need the following technical estimates on the fractional Brownian motion.

\begin{proposition}
Given a fractional Brownian motion there exists $C$ such that 
\begin{equation} \label{fBmBound}
\int_{\R^m} \prod_{j=1}^m |v_j|^k \exp \left\{ - \frac12 Var \left(  \sum_{j=1}^m v_j  B_{s_j} \right)  \right\} dv_1 \dots dv_m \leq C^m \sqrt{(km)!} \prod_{j=1}^m |s_{j} - s_{j-1}|^{-H(1 + 2k)}
\end{equation}
for all $(s_1, \dots, s_m) \in \Delta^{(m)}(0,T)$, and we read $s_0 = 0$.
\end{proposition}

\begin{proof}
Define the matrix $A_{i,j} = E[B_{s_i} B_{s_j}]$, let $X \sim \mathcal{N}(0, A^{-1})$ and denote by $\tilde{X}$ the $km$-dimensional Gaussian vector $$
\tilde{X}_i = X_j \textrm{ for } (j-1)k + 1 \leq i \leq jk.
$$ 

Rewrite the right hand side of \eqref{fBmBound} as 
\begin{align*}
(2 \pi)^{m/2} |A|^{-1/2}   E [ \prod_{j=1}^m |X_j|^k]  & = (2 \pi)^{m/2} |A|^{-1/2}   E [ \prod_{i=1}^{km} |\tilde{X}_i|] \\
 & \leq (2 \pi)^{m/2} |A|^{-1/2}  \left( \sum_{\sigma \in S_{km}} \prod_{i=1}^{km} E[ \tilde{X}_i \tilde{X}_{\sigma(i)} ]   \right)^{1/2} \\ 
 &  \leq (2 \pi)^{m/2} |A|^{-1/2}  \left( \sum_{\sigma \in S_{km}} \prod_{i=1}^{km} E[ \tilde{X}_i^2  ]^{1/2} E[ \tilde{X}_{\sigma(i)}^2 ]^{1/2}   \right)^{1/2} \\ 
  &  = (2 \pi)^{m/2} |A|^{-1/2}  \left( \sum_{\sigma \in S_{km}} \prod_{j=1}^m E[ X_j^2  ]^k  \right)^{1/2}   = (2 \pi)^{m/2} |A|^{-1/2}  \left( (km)!  \prod_{j=1}^m E[ X_j^2  ]^k  \right)^{1/2}, 
 \end{align*}
 where we have used \cite[Theorem 1]{LiWei} in the first inequality.
Then we get from \eqref{CramersRule} that
$$
(A^{-1})_{j,j} 
\geq c |s_{j+1} - s_j|^{2H} \wedge |s_{j} - s_{j-1}|^{2H} \geq c |s_{j} - s_{j-1}|^{4H} 
$$
where we have used \eqref{LocalNonDeterminism} and $|s_{j+1} - s_{j}| \leq 1$ in the two last steps, respectively.
Using \eqref{GaussianCovariance} and \eqref{LocalNonDeterminism} we get that 
$$
|A|^{-1/2} \leq c^{-m} \prod_{j=1}^m |s_j - s_{j-1}|^{-H}
$$
The result follows.  
%
\end{proof}

Let us mention that the fractional Brownian motion can be lifted to a rough path. This was first done in \cite{Unterberger}, but we shall refer to \cite{NualartTindel} for a different construction where the authors construct the iterated integrals using a Stratonovich-Volterra-type representation.

\begin{theorem}[Theorem 1.1. in \cite{NualartTindel}]
Let $B$ be a fractional Brownian motion admitting the representation (\ref{fBmRepresentation}). For $1 \leq n \leq \lfloor \frac{1}{H} \rfloor$ define
    $$
B^{(n)} : \Delta^{(2)}(0,T) \rightarrow (\R^d)^{\otimes n}
$$
component wise, i.e. for any tuple $\{ i_1, \dots i_n \}$ in $\{ 1, \dots, d\}$, as the Stratonovich iterated integral
\begin{align*}
  \langle B^{(n)}_{st}, & e_{i_1} \otimes \dots \otimes  e_{i_n} \rangle = 
   \sum_{j=1}^n (-1)^{j-1}  \int_{A^n_j} \prod_{l=1}^{j-1} K(s,r_l) [ K(t,r_j) - K(s,r_j)] \prod_{l=j+1}^n K(t,r_l) \circ dW^{i_1}_{r_1} \dots \circ dW^{i_n}_{r_n} \\
\end{align*}
where 
\begin{align*}
  A^n_j  :=  \{ (r_1, & \dots r_n) \in [0,t]^n : r_j = \textrm{min}(r_1, \dots, r_n),  r_1 > \dots > r_{j-1} \textrm{ and } r_{j+1} < \dots < r_n \} . \\
\end{align*}
Then there exists a set $\Omega_{\B}$ with full measure such that
$$
\B_{st} := (1, B_t - B_s, B^{(2)}_{st} , \dots B^{(\lfloor 1/H \rfloor)}_{st})
$$
satisfies  (\ref{Chen}) and (\ref{symmetryCondition})  on $\Omega_{\B}$. Moreover, for $\gamma <H$ we have $|B^{(n)}_{st}| \lesssim |t-s|^{\gamma n}$.
\end{theorem}

Assume now that $H$ is such that $\frac{1}{H}$ is not an integer. We can choose $ \gamma < H$ such that $\lfloor \frac{1}{\gamma} \rfloor = \lfloor \frac{1}{H} \rfloor$, and from the above theorem we have, $P$-a.s., $\B \in \mathscr{C}^{\gamma}_g$.

Let us remark that for $H \in (\frac{1}{4} , \frac{1}{2})$ there exists a lift of $B$ to a rough path building the iterated integral from linear interpolation of $B$. For the method of the current paper to work we need smaller $H$, see Section \ref{fBmSDEChapter}. When $H \in (0, \frac{1}{4})$ the dyadic interpolation fails to give a converging sequence of rough paths, see \cite{CoutinQian}. Nevertheless, the construction in \cite{NualartTindel} gives a geometric rough path so that we may approximate $\B$ by a sequence of lifted smooth paths, in the rough path topology. 

\section{Fractional Brownian motion SDE's} \label{fBmSDEChapter}

For this section we shall study a SDE driven by an additive fractional Brownian motion, i.e.
\begin{equation} \label{SDE}
\phi_t(x) = x + \int_0^t b(r,\phi_r(x))dr + B_t .
\end{equation}
Existence and uniqueness of a solution to this equation under low regularity on $b$ was recently proved in \cite{BNP} as demonstrated in the next Proposition. For proofs the reader is referred to \cite{BNP}.

\begin{proposition}[Theorem 4.1 and Corollary 4.8 in \cite{BNP}]  \label{compactSequence}
  Assume $H < \frac{1}{2(2d+1)}$.
  Let $\{b_n\}_{n \geq 0} \subset C^{\infty}_c([0,T] \times \R^d ; \R^d)$ be a sequence of functions such that
$$
\sup_{n \geq 0}  \left( \| b_n \|_{\infty,1} \vee  \| b_n \|_{\infty}  \right) < \infty .
$$

Denote by $\phi_n(t,x)$ the solution to (\ref{SDE}) when $b$ is replaced by $b_n$. Then for fixed $(t,x) \in [0,T] \times \R^d$ the sequence is $\phi_n(t,x)$ is relatively compact in the \emph{strong} topology of $L^2(\Omega)$.

Furthermore, if $ \lim_{n \rightarrow \infty}b_n(t,x) = b(t,x) $ for almost all $(t,x) \in [0,T] \times \R^d$ for $b \in L^1(\R^d; L^{\infty}([0,T];\R^d)) \cap  L^{\infty}([0,T] \times \R^d;\R^d)$ then $\phi_n(t,x)$ is converging for every $(t,x) \in [0,T] \times \R^d$ to the unique solution of (\ref{SDE}).

\end{proposition}

The proof of this Proposition relies on a compactness criterion from \cite{DPMN} based on Malliavin calculus. Without going into too much detail there is compactness in $L^2(\Omega)$ if we can bound the Malliavin derivative of $\phi_n(t,x)$ by a constant depending only on $\| b_n \|_{L^1(\R^d; L^{\infty}([0,T];\R^d))} \vee  \| b_n \|_{L^{\infty}([0,T] \times \R^d;\R^d)}$. 

Once one has strong convergence, one can use a somewhat standard trick, see e.g. \cite{GyongyPardoux} or \cite{NualartOuknine}, to show that $\int_0^t b_n(r,\phi_n(r,x))dr \rightarrow \int_0^t b(r,\phi_r(x))dr$ which gives that the limit solves (\ref{SDE}).

\vspace{.2cm}

Furthermore the following result shows how the fBm regularizes the flow of (\ref{SDE}).

\begin{lemma}[Theorem 5.1 in \cite{BNP}]
Assume $H < \frac{ 1}{(d-1 + 2k)}$ and let $p,k$ be integers, $p \geq 2$, $k \geq 1$. There exists an increasing function $C: [0,\infty) \rightarrow [0,\infty)$ only depending on $H,d,p$ and $k$ such that
$$
\sup_{t \in [0,T], x \in \R^d} E \left[ \left| \nabla^k \phi_n(t,x) \right|^p \right] \leq C(  \| b_n \|_{\infty,1} \vee  \| b_n \|_{\infty}).
$$
\end{lemma}

Using the two previous results together with weak compactness in $L^2(\Omega;W^{k,p}(U))$ for an open and bounded $U \subset \R^d$ we get the following result.

\begin{theorem}[Theorem 5.2 in \cite{BNP}]
Assume $H < \left( \frac{1}{2(2d+1)} \wedge \frac{1}{2(d-1 + 2k)} \right)$ and $b \in L^1(\R^d; L^{\infty}([0,T];\R^d)) \cap  L^{\infty}([0,T] \times \R^d;\R^d)$. For every open and bounded $U \subset \R^d$ the solution to (\ref{SDE}) is $k$-times weakly differentiable in the sense that
$$
\phi_t \in L^2(\Omega; W^{k,p}(U))
$$
for every $p > 1$. Moreover, $\phi_n(t)$ converges to $\phi_t$ in the weak topology of $L^2(\Omega; W^{k,p}(U))$.

\end{theorem}

\subsection{The one-dimensional case}

In this section we include a proof of Proposition \ref{compactSequence}  when $d=1$ and $H < \frac{1}{6}$. 
From \cite{NualartOuknine} it is already known that there exists a unique strong solution to this equation when $b$ of linear growth. From \cite{NualartOuknine} it also becomes clear why the proof is simpler when $d=1$ - one can use comparison of SDE's to generate the strong convergence as indicated in Section \ref{singularSDE}.

We shall restrict our attention to when $b$ is bounded and integrable, but we are interested in how the solution depends on the initial value $x$.
More specifically we will show the following.

\begin{theorem} \label{sobolevFlowTheorem}
Assume $b \in L^1(\R^d; L^{\infty}([0,T];\R)) \cap L^{\infty}([0,T] \times \R)$ . If $H < \frac{1}{6}$ there exists a unique strong solution to (\ref{SDE}). Moreover the mapping $(x \mapsto \phi_t(x))$ is weakly differentiable in the sense that for fixed $t$ we have
$$
\phi_t (\cdot) \in L^2(\Omega; W^{1,p}(U))
$$
for all open and bounded $U \subset \R$.

\end{theorem}

This theorem is proved in three steps. In the first step we establish an integration by parts formula for the fractional Brownian motion. In the second step we assume $b$ is smooth and has compact support. It is then well known that $\phi_t(\cdot)$ is smooth, and we use the integration by parts formula to bound $\| \phi_t\|_{L^2(\Omega;W^{1,p}(U))}$  independently of $b'$. In the third step we approximate a general $b$ by smooth functions. We use comparison to generate strong convergence in $L^2(\Omega)$ of the corresponding sequence of solutions. From step one and two we can bound the sequence in $L^2(\Omega;W^{1,p}(U))$ and argue via weak compactness to prove Theorem \ref{sobolevFlowTheorem}.

\subsubsection{An integration by parts formula} \label{ibpFormula}

The purpose of this section is to prove a integration by parts type formula involving a random variable inspired by local time calculus. More specifically, we have
\begin{equation} \label{integrationByParts}
\int_0^t b'(s,B_s) ds = - \int_{\R} \Lambda^b(t,y) dy \hspace{.4cm} P-a.s.
\end{equation}
where 
\begin{equation} \label{LambdaRepresentation}
\Lambda^b(t,y) = (2\pi)^{-1} \int_{\R} \int_0^t b(s,y) iu e^{ -i u (B_s - y) }  ds du .
\end{equation}

We start by \emph{defining} $\Lambda^b(t,z)$ as above, and prove that it is a well defined element of $L^p(\Omega)$ for every $p > 1$. 

\begin{lemma} \label{momentEstimate}
Assume $b$ is bounded. Then $\Lambda^b(t,y)$ exists and all moments are integrable provided $H<\frac13$. More precisely if $m$ is an even integer
$$
E[ |\Lambda^b(t,y)|^m ] \leq \frac{C^m \|b(\cdot, y)\|^m_{\infty} m! \sqrt{m!}}{ \Gamma( m(1-3H) + 1)} .
$$
\end{lemma}

\begin{proof}

Since we assume $m$ is an even integer, we may write
\begin{align*}
E[ |\Lambda^b(t,y)|^m & = (2\pi)^{-m} E  | \int_{\mathbb{R}}\int_0^t  b(s,y) iu \exp\{ - iu  (B_s - y) \} ds du |^m \\ 
& = (2 \pi)^{-m} m! \int_{\Delta^m(0,t)}  \int_{\mathbb{R}^{m} } b^{\otimes m}(s,y)\prod_{j=1}^m iu_j E[\exp\{ - iu_j  (B_{s_j} - y) \}]    du ds \\ 
& \leq  (2 \pi)^{-m} m! \int_{\Delta^m(0,t)}  \int_{\mathbb{R}^{m} } |b^{\otimes m}(s,y)| \prod_{j=1}^m |u_j| \exp\{ - \frac12 Var (\sum_{j=1}^m u_j  B_{s_j} ) \}]    du ds 
\end{align*}
where for notational convenience we have used $B_{s_0} = y$ and $v_{m+1} = 0$, $ds = ds_1 \dots ds_m$, $du = du_1 \dots du_m$ and $b^{\otimes m}(s,y) := \prod_{j=1}^m b(s_j,y)$.
Using \eqref{fBmBound} the above is bounded by
$$
C^m  m! \sqrt{m!} \|b(\cdot, y)\|^m_{\infty} \int_{\Delta^{(m)}(0,t)} \prod_{j=1}^m |s_j - s_{j-1}|^{-3H} ds \leq \frac{C^m m! \sqrt{m!} \|b(\cdot, y)\|^m_{\infty}}{\Gamma( (1 - 3H)m + 1)} .
$$
\end{proof}

From \eqref{LambdaRepresentation} we see that $ supp \Lambda^{b}(t, \cdot) \subset \bigcup_{ s \leq t} supp b(s, \cdot)$. In particular, if the latter set is bounded, $\Lambda^b(t,\cdot)$ is integrable $P$-a.s.

It remains to show that $\Lambda^b$ satisfies the integration by parts formula (\ref{integrationByParts}). Notice that one has to be careful interchanging the order of integration in (\ref{LambdaRepresentation}). Indeed, if $b=1$, one should think of $\int_{\R} iu e^{ -i u (B_s - y) } du = - \partial_y  \delta_{B_s}(y)$ where $\delta_{B_s}(y)$ is the Donsker-Delta of $B_s$, which is not a random variable in the usual sense (one could introduce the Donsker-Delta as a generalized random variable in the sense of White Noise theory, but we shall avoid this).

To circumvent this difficulty we define an approximating sequence
$$
\Lambda^b_K(t,y) := (2\pi)^{-1} \int_{-K}^K \int_0^t b(s,y) iu e^{ -i u \cdot  (B_s - y) }  ds du. 
$$
It is immediate that 
$$
|\Lambda^b_K(t,y)| \leq C_K  \int_0^t |b(s,y)|  ds,  
$$
for an appropriate constant, so that $\Lambda^b_K(t,\cdot)$ is integrable if $\int_{\R}\int_0^t|b(s,y)|dsdy < \infty$. One can show that $\Lambda^b_K(t,y) \rightarrow \Lambda^b(t,y)$ in, say, $L^2(\Omega)$ for all $t$ and $y$. To see this the reader is invited to modify the above proof to see that 
\begin{equation*}
E[|\Lambda_K^b(t,y) - \Lambda^b(t,y)|^2] \leq C \|b(\cdot, y)\|^2_{\infty} \int_{\Delta^2(0,t)} \int_{\R^2} 1_{\{|u_1| > K\} } 1_{\{|u_2| > K\} } |u_1| |u_2| e^{-\frac{1}{2} Var( u_1 B_{s_1} + u_2 B_{s_2})  }du ds
\end{equation*}
which converges to zero as $K \rightarrow \infty$. In the above $C$ is a constant that is independent of $K$. Now we have
\begin{align*}
\int_{\R} \Lambda^b_K(t,y) dy & = (2\pi)^{-1/2} \int_{-K}^K \int_0^t (\mathcal{F}^{-1}b)(s,u) iu e^{-iu \cdot X_s} ds du \\
& =\int_0^t (2 \pi)^{-1/2} \int_{-K}^K (\mathcal{F}^{-1}b)(s,u) iu e^{-iu \cdot X_s} du ds.
\end{align*}
Provided $b(s,\cdot) \in \mathcal{S}(\R)$ we have 
\begin{align*}
\lim_{K \rightarrow \infty} (2\pi)^{-1/2} \int_{-K}^K (\mathcal{F}^{-1}b)(s,u) iu e^{-iuX_s} du & = (2\pi)^{-1/2} \int_{\R} (\mathcal{F}^{-1}b)(s,u) iu e^{-iuX_s} du \\
= \mathcal{F}(iu(\mathcal{F}^{-1}b)(s,u))(X_s) & =   -b'(s,X_s)
\end{align*}
thus proving (\ref{integrationByParts}).

We summarize these considerations.

\begin{lemma} \label{ibpProposition}
Let $b: [0,T] \times \R \rightarrow \R$ be such that $b(s,\cdot)$ is smooth for every $s$ and $\bigcup_{s \leq T} supp b(s, \cdot)$ is a bounded set. Then (\ref{integrationByParts}) holds on a set of measure 1.

\end{lemma}

We can however extend (\ref{integrationByParts}) to $b$ bounded and differentiable.

\begin{lemma} \label{supportInclusion}
Assume $b \in L^{\infty}( [0,T] ; C^1_b(\R))$. Then (\ref{integrationByParts}) holds for $b$ and we have $P$-a.s.
$$
supp \Lambda^{b}(t,\cdot) \subset [-B_t^*, B^*_t]
$$
where $B_t^* := \sup_{0 \leq s \leq t} |B_s|$.
\end{lemma}

\begin{proof}
Assume first that $b$ satisfies the assumptions of Lemma \ref{ibpProposition}, and let $\phi \in C_c^1(\R)$. From (\ref{LambdaRepresentation}) we have $\Lambda^{\phi b}(t,y) = \phi(y) \Lambda^b(t,y)$. 
Consequently, using (\ref{integrationByParts})
\begin{align*}
\int_{\R} \phi(y) \Lambda^b(t,y) dy & = \int_{\R} \Lambda^{\phi b}(t,y) dy  = -\int_0^t (\phi(B_s) b(s, B_s))' ds \\
& = -\int_0^t \phi'(B_s) b(s,B_s) ds - \int_0^t \phi(B_s) b'(s,B_s) ds ,
\end{align*}
so that for all $\phi \in C^1_c(\R)$ such that $supp \phi  \cap  [-B_t^*, B^*_t] = \emptyset$, we have $\int_{\R} \phi(y) \Lambda^b(t,y)dy = 0$. In particular, $\Lambda^b(t,\cdot)$ has compact support independent of $b$ $P$-a.s.


From linearity of $b \mapsto \Lambda^b$ and Lemma \ref{momentEstimate} we may approximate a general $b$ by smooth, compactly supported functions. The result follows by elementary calculations.
\end{proof} 

Using $\Lambda^{\phi b}(t,y) = \phi(y) \Lambda^{b}(t,y)$ as in the above proof we get that if $b$ is time homogeneous, $\Lambda^b(t,y) = b(y) \partial_y L^B(t,y)$ where $L^B(t,y)$ denotes the local time of the fractional Brownian motion (which is well known to be differentiable when $H < \frac{1}{3}$, see \cite{GemanHorowitz}).

\begin{proposition} \label{integralMoments}
There exists a constant $C > 0$ such that for all even integers $m$
$$
E\left[ \left(  \int_{\mathbb{R}} | \Lambda^{b}(t,y) | dy  \right )^m \right] \leq \frac{C^m  \|b\|^m_{\infty, 1} m^{m/2} \sqrt{(2m)!}}{ \sqrt{\Gamma ( m(1-3H) + 1) }}.
$$
\end{proposition}

\begin{proof}

We write 
\begin{align*}
E\left[ \left(   \int_{\mathbb{R}} | \Lambda^b(t,y) | dy   \right )^m \right] & = \int_{\R^m} E[ \prod_{j=1}^m  | \Lambda^b(t,y_j) | ] dy_1 \dots dy_m 
 \leq \int_{\R^m} \prod_{j=1}^m \prod_{j=1}^m E[ | \Lambda^b(t,y_j) |^{m} ]^{1/m} dy_1 \dots dy_m \\
& \leq  \frac{C^m  m! \sqrt{m!}}{ \Gamma((1 - 3H)m + 1)} \int_{\R^m} \prod_{j=1}^m \|b(\cdot, y_j) \|_{\infty} dy_1 \dots dy_m
  = \frac{C^m  m! \sqrt{m!}}{ \Gamma((1 - 3H)m + 1)}  \|b \|^m_{1, \infty} 
\end{align*}
for an appropriate constant $C$, where we have used Lemma \ref{momentEstimate}.
\end{proof}

\subsubsection{Derivative free estimates}

In this section we assume that $b \in L^{\infty} ([0,T] ; C^1_c(\R))$ and denote by $\phi_{\cdot}(x)$ the solution to (\ref{SDE}). It is well known that $\phi_t(\cdot)$ continuously differentiable, and we have 
\begin{align} 
\partial_x \phi_t(x)  & = 1 + \int_0^t b'(r,\phi_r(x)) \partial_x \phi_r(x) dr \label{ODE} \\
  & = \exp \{ \int_0^t b'(r,\phi_r(x)) dr \} \label{ODESolution} .
\end{align}

We are ready to prove our main estimate on SDE's.

\begin{theorem} \label{derivativeFreeBound}
There exists an increasing continuous function  $C: [0,\infty) \rightarrow [0, \infty)$ such that for all $b \in L^{\infty} ([0,T] ; C^1_b(\R))$
$$
\sup_{t \in [0,T], x \in \R}  E \left[ \left( \partial_x \phi_t(x) \right)^2 \right] \leq C(\|b\|_{\infty} \wedge \|b\|_{\infty,1}),
$$
where $\phi_{\cdot}(x)$ is the unique solution of (\ref{SDE}) driven by $b$.
\end{theorem}

\begin{proof}


Set $\theta_t := \left(K_H^{-1}\left(\int_0^{\cdot} b(r,\phi_r(x))dr\right) \right) (t)$ and consider the Dol\'{e}ans-Dade exponential
\begin{align*}
Z & := \exp \left\{ \int_0^T \theta_s dW_s -\frac{1}{2} \int_0^T \theta_s^2 ds  \right\}.
\end{align*}
Define the measure $\tilde{P}$ by 
\begin{align*}
d\tilde{P} & := Z dP .
\end{align*}
Then $\tilde{P}$ is a probability measure and under $\tilde{P}$ the solution $\{\phi_t(x)\}_t$ is a fractional Brownian motion starting in $x$. From (\ref{ODESolution}) we get
\begin{align*}
E[(\partial_x\phi_t(x))^2]  & = E[ \exp \{ 2 \int_0^t b'(r,\phi_r(x)) dr \}] = \tilde{E}[ \exp \{ 2 \int_0^t b'(r,\phi_r(x)) dr \} Z^{-1} ] \\
 & \leq \left( \tilde{E}[ \exp \{ 4 \int_0^t b'(r,\phi_r(x)) dr \}  ] \right)^{1/2} \left( \tilde{E}[ Z^{-2} ] \right)^{1/2}. 
\end{align*}
Now we write 
\begin{align*}
\tilde{E}[ \exp \{ 4 \int_0^t b'(r,\phi_r(x)) dr \}  ] & = E[ \exp \{ 4 \int_0^t b'(r,x + B_r) dr \}  ]  = E[ \exp \{ 4  \int_{\R} \Lambda^b(t,y) dy \}  ] \\
& = \sum_{m \geq 0} \frac{ 4^m E\left[ \left(  \int_{\R} \Lambda^b(t,y) dy \right)^m \right] }{m!}  \leq \sum_{m \geq 0} \frac{(4 \|b\|_{1, \infty})^m  C^m (2m!)^{1/4} \sqrt{(2m)!}}{ m! \sqrt{\Gamma( (1-3H)2m + 1)}}  \\
 & =: \tilde{C}( \|b\|_{ \infty,1}) 
\end{align*}
which converges by Stirling's formula.

From Theorem \ref{girsanov} we know that we can bound $\tilde{E}[Z^{-2}]$ by a function depending on $\|b\|_{\infty}$. The result follows.
\end{proof}

\subsubsection{Singular SDE's} \label{singularSDE}

For this section we shall consider a bounded and measurable $b: [0,T] \times \R \rightarrow \R$ and the corresponding SDE (\ref{SDE}). As indicated above we shall use an approximation $b_n$ of $b$ and comparison to generate strong convergence in $L^2(\Omega)$. The technique is somewhat classical, and we refer to \cite{NualartOuknine} for a proof, but let us briefly explain the idea:

Let $b$ be bounded and measurable and define for $n \in \mathbb{N}$
$$
b_n(t,x) := n\int_{\mathbb{R}} \rho(n(x-y)) b(t,y) dy 
$$
where $\rho$ is a non-negative smooth function with compact support in $\mathbb{R}$ such that $\int_{\mathbb{R}} \rho(y)dy = 1$.

We let 
$$
\tilde{b}_{n,k} := \bigwedge_{j=n}^k b_j, \hspace{1cm} n \leq k
$$
and 
$$
\mathcal{B}_n = \bigwedge_{j=n}^{\infty}  b_j,
$$
so that $\tilde{b}_{n,k}$ is Lipschitz. Denote by $\tilde{\phi}_{n,k}(t,x)$ the unique solution to (\ref{SDE}) when we replace $b$ by $\tilde{b}_{n,k}$. Then one can use comparison to show that 
$$
\lim_{k \rightarrow \infty} \tilde{\phi}_{n,k}(t,x) = \phi_n(t,x),  \hspace{1cm} \textrm{in } L^2(\Omega)
$$ 
where $\phi_n(t,x)$ solves (\ref{SDE}) when we replace $b$ by $\mathcal{B}_n$. Furthermore, 
$$
\lim_{n \rightarrow \infty} \phi_{n}(t,x) = \phi_t(x),  \hspace{1cm} \textrm{in } L^2(\Omega)
$$ 
where $\phi_t(x)$ is a solution to (\ref{SDE}). For details see \cite{NualartOuknine}. 

We are ready to prove the main result of the section.

\begin{proof}[Proof of \ref{sobolevFlowTheorem}]
Let $U \subset \R$ be open and bounded. We know from the discussion above that $\phi_{n}(t,x) \rightarrow \phi_t(x)$ in $L^2(\Omega)$. From Theorem \ref{derivativeFreeBound} plus elementary bounds we see that $\phi_n(t, \cdot)$ is bounded in $L^2(\Omega; W^{1,2}(U))$. Consequently we may extract a subsequence $\{ \phi_{n_k}(t, \cdot) \}_{k \geq 1}$ converging to an element $f_t$ in the weak topology of $L^2(\Omega; W^{1,2}(U))$. Let $A \in \mathcal{F}$ and $\eta \in C^{\infty}(U)$. Using strong convergence coupled with weak convergence we get
\begin{align*}
E[ 1_A \int_U \phi_t(x) \eta'(x) dx ]   = \lim_{k \rightarrow \infty} E [ 1_A \int_U \phi_{n_k}(t,x) \eta'(x) dx ]  &  =- \lim_{k \rightarrow \infty} E [ 1_A \int_U \partial_x \phi_{n_k}(t,x) \eta(x) dx ] \\
& =- E [ 1_A \int_U \partial_x f_t(x) \eta(x) dx ].
\end{align*}
Consequently we have $\int_U \phi_t(x) \eta'(x) dx = - \int_U \partial_x f_t(x) \eta(x) dx$ on some $\Omega_{\eta} \in \mathcal{F}$ such that $P(\Omega_{\eta}) =1$. Let now $\Omega^*$ be the intersection of a countable, dense in $W^{1,2}(U)$, set of $\eta$ such that the above integration by parts formula holds. It is clear that $P(\Omega^*)=1$ and that $\phi_t$ is weakly differentiable on this set. The result follows.
\end{proof}

\begin{remark} \label{inverseConvergence}
For fixed $t_0 > 0$, consider the equation
$$
\psi_t^{t_0}(y) = y - \int_0^t b(t_0 - r, \psi_r^{t_0}(y)) dr  - (B_{t_0} - B_{t_0 -t}) .
$$
Since the fractional Brownian motion has stationary increments the above equation is on the same form as (\ref{SDE}) and we may apply the same machinery to obtain a sequence $\psi_t^{t_0, n}$ of corresponding smooth flows that converges in the weak topology of $L^2(\Omega; W^{1,p}(U))$ and $\psi_t^{t_0, n}(x)$ converges in the strong topology of $L^2(\Omega)$ to the solution of the above equation.

We have $\psi_{t_0}^{t_0} = \phi_{t_0}^{-1}$, so that $\phi_{t_0}$ is invertible with a Sobolev-differentiable inverse.

Let now $f \in C^1_b(\R)$. For every $n \in \mathbb{N}$ we have
$$
\partial_x f( \psi_t^{t,n}(x)) = f'( \psi_t^{t,n}(x)) \partial_x  \psi_t^{t,n}(x)
$$
which is bounded in any $L^p(U)$ for $p > 1$, $U$ open and bounded. Consequently, there is a weakly converging subsequence which by uniqueness must converge weakly in $L^p(U)$ and we have
$$
\partial_x f( \psi_t^{t}(x)) = f'( \psi_t^{t}(x)) \partial_x  \psi_t^{t}(x) \vspace{.5cm} \textrm{ for almost all } x \in \R.
$$
since $\psi_t^{t,n}(x) \rightarrow \psi_t^t(x)$ strongly in $L^2(\Omega)$.
When $b$ is time-homogenuous we have the following representation
$$
\phi_t^{-1}(y) = y - \int_0^t b( \phi_r^{-1}(y))dr  - B_t . 
$$

\end{remark}

\subsection{Local time of the flow} \label{localTimeOfTheFlow}
We now return to the general case of $d \geq 1$.

In this section we develop a local time theory for the solutions $\phi_t(x)$ of (\ref{SDE}). Assuming we have a solution to $\phi_t(x)$, the results here will rely only on Girsanov's theorem \ref{girsanov} meaning we only use boundedness of $b$. Let now $Q: [0,T] \times \R^d \rightarrow \R$ be given and define $q = D^{\alpha}Q$ for some multiindex, $\alpha = (\alpha^{(1)}, \dots \alpha^{(d)})$. The main objective of this section is to prove that there exists a random field $\Lambda_{\alpha}^{\phi(x),Q}$ on $[0,T] \times \R^d$ such that
$$
\int_0^t q(s,\phi_s(x)) ds = (-1)^{|\alpha|} \int_{\R^d} \Lambda_{\alpha}^{\phi(x),Q}(t,y) dy,
$$
and that the right hand side above can be bounded in terms of $Q$.
Motivated by the previous subsection, we define
$$
\Lambda_{\alpha}^{\phi(x),Q}(t,y) = (2\pi)^{-d} \int_{\R^d} \int_0^t (iu)^{\alpha} Q(s,y) \exp\{ -i u \cdot (\phi_s(x) - y) \} ds du.
  $$
  We denote by $\Lambda_{\alpha}^{Q}(t,y)$ the random field obtained by choosing $B$ instead of $\phi(x)$ in the above definition. Note that from Girsanov's theorem we have
  $$
E[ f(\Lambda_{\alpha}^{\phi(x),Q}(t,y)) ] = E[ f(\Lambda_{\alpha}^{Q}(t,y))  \xi_T]
    $$
    for any $f$ such that the above expressions exists and $\xi_T$ was defined in Theorem \ref{girsanov} and we have $E[ |\xi_T|^2 ] \leq C(\|b\|_{\infty})$
    
  We get a similar result as Lemma \ref{momentEstimate}.

\begin{lemma}
  Assume $Q$ is bounded and $H < \frac{1}{d + 2 |\alpha| }$. We have the following moments estimates on $\Lambda_{\alpha}^{Q}$

\begin{equation}
E[ | \Lambda_{\alpha}^Q (t,y)|^m ] \leq \frac{C^m m! \prod_{k=1}^d  \sqrt{ (m \alpha^{(k)})!} }{ \Gamma( m ( 1 - H(d + 2 |\alpha|)) + 1 }
\end{equation}
where $C = C(\alpha,H)$ does not depend on $m$ or $Q$.

\end{lemma}

\begin{proof}
The proof follows the same lines as in the proof of Lemma \ref{momentEstimate}.
  Begin by writing
\begin{align*}
 E \left[ | \Lambda_{\alpha}^Q (t,y)|^{m} \right] & = \frac{m!}{(2 \pi)^{dm}} \int_{(\R^d)^m} \int_{\Delta^{(m)}(0,t)} \prod_{j=1}^m (iu_j)^{\alpha} Q(s_j,y) \exp\{ - \frac12 Var( u_j \cdot (B_{s_j} - y) \}  ds du\\
& \leq \frac{m!}{(2 \pi)^{dm}} \int_{(\R^d)^m} \int_{\Delta^{(m)}(0,t)} \prod_{j=1}^m  \prod_{k=1}^d |u_j^{(k)}|^{\alpha^{(k)}} |Q(s_j,y)| \exp\{ - \sum_{k=1}^d \frac12 Var(  \sum_{j=1}^m  u_j^{(k)}  B_{s_j}^{(1)}   ) \}  ds du\\ 
& \leq \frac{m! \|Q(\cdot, y)\|_{\infty}^m }{(2 \pi)^{dm}} \int_{\Delta^{(m)}(0,t)} \prod_{k=1}^d  \int_{\R^m}  \prod_{j=1}^m |u_j^{(k)}|^{\alpha^{(k)}}  \exp\{ - \frac12 Var(  \sum_{j=1}^m  u_j^{(k)}  B_{s_j}^{(1)}   ) \}  ds du_1^{(k)} \dots du_m^{(k)} \\ 
\end{align*}
where we have used the independence of the components of $B$ in the second line. Using \eqref{fBmBound}, the above is bounded by
\begin{align*}
C^m m! \|Q(\cdot, y)\|_{\infty}^m  & \prod_{k=1}^d  \sqrt{ (m \alpha^{(k)})!}  \int_{\Delta^{(m)}(0,t)} \prod_{k=1}^d   |s_j - s_{j-1}|^{ - H ( 1 + 2 \alpha^{(k)})}  ds  \\ 
 & \leq \frac{C^m m! \prod_{k=1}^d  \sqrt{ (m \alpha^{(k)})!} \|Q(\cdot, y)\|_{\infty}^m }{ \Gamma( m ( 1 - H(d + 2 |\alpha|)) + 1 } 
\end{align*}
provided $ H < \frac{1}{d + 2 |\alpha|}$.

\end{proof}

Using Theorem \ref{girsanov} we get

\begin{corollary} \label{flowLocalTime}
  Let $Q$, $H$ and $\alpha$ be as in the previous lemma. There exists a constant $C = C(\|b\|_{\infty},H,d,\alpha)$ such that
  $$
  E[ | \Lambda_{\alpha}^{\phi(x),Q} (t,y)|^m ] \leq \frac{C^m \sqrt{ (2m)! \prod_{k=1}^d  \sqrt{ (2m \alpha^{(k)})!}}  \|Q(\cdot, y)\|_{\infty}^m }{ \sqrt{ \Gamma( 2m ( 1 - H(d + 2 |\alpha|)) + 1 } }
  $$

\end{corollary}

If we assume integrability of $Q$ in the spatial variable we see that we can define the stochastic process $\int_{\R^d} \Lambda_{\alpha}^{\phi(x),Q}(t,z) dz$.

\begin{lemma} \label{exponentialBoundsLambda}
If we assume $Q \in L^1(\R^d; L^{\infty}([0,T]))$, $|\alpha| \leq 1$ and $H< \frac{1}{d + 2}$ we have 
$$
E[ \exp\{ \int_{\R^d} | \Lambda_{\alpha}^{\phi(x),Q}(t,z)| dz \} ] \leq C( \|Q\|_{\infty,1} \wedge   \|b\|_{\infty} )
$$
where $C$ is an increasing function.

\end{lemma}

\begin{proof}
Begin by writing
\begin{align*}
E\left[  \left(\int_{\R^d} | \Lambda_{\alpha}^{\phi(x),Q}(t,z)| dz \right)^m  \right] & = \int_{(\R^d)^m} E\left[ \prod_{j=1}^m | \Lambda_{\alpha}^{\phi(x),Q}(t,z_j)| \right] dz_1 \dots dz_m \\
&  \leq \int_{(\R^d)^m}  \prod_{j=1}^m E\left[| \Lambda_{\alpha}^{\phi(x),Q}(t,z_j)|^m \right]^{1/m} dz_1 \dots dz_m \\
& \leq \frac{C^m \sqrt{ (2m)! \prod_{k=1}^d  \sqrt{ (2m \alpha^{(k)})!}} }{ \sqrt{ \Gamma( 2m ( 1 - H(d + 2 |\alpha|)) + 1 } } \int_{(\R^d)^m} \prod_{j=1}^m \| Q(\cdot, z_j)\|_{\infty}  dz_1 \dots dz_m \\
& \leq \frac{C^m \sqrt{ (2m)!  \sqrt{ (2m)!}} }{ \sqrt{ \Gamma( 2m ( 1 - H(d + 2 )) + 1 } }\|Q\|^m_{\infty,1}
\end{align*}
where $C$ is as in Corollary \ref{flowLocalTime}. We get
\begin{align*}
E [ \exp\{ \int_{\R^d} | \Lambda_{\alpha}^{\phi(x),Q}(t,z)| dz \} ] &  = \sum_{m \geq 0}  (m!)^{-1} E\left[  \left(\int_{\R^d} | \Lambda_{\alpha}^{\phi(x),Q}(t,z)| dz \right)^m  \right] \\
& \sum_{m \geq 0 }\frac{ C^m \sqrt{ (2m)! \sqrt{(2m)!} }}{ \sqrt{\Gamma (2m (1 - H(d+2)) + 1)} m!} \|Q\|^m_{\infty,1}\\
\end{align*}
which converges as long as $H< \frac{1}{d + 2}$ by Stirling's formula.
\end{proof}

We now proceed to prove stability of the vector field $\Lambda_{\alpha}^{\phi(x),Q}$ in both $Q$ and $\phi$ in the following way.

\begin{remark}
We shall need stability of the mapping $(\phi(x), Q) \mapsto \int_{\R^d} \Lambda^{\phi(x), Q}(t,z) dz$, but we only need continuity in each variable separately. If $\phi^{\epsilon}_{\cdot}(x)$ converges to $\phi_{\cdot}(x)$ in, say, Lebesgue measure over $[0,T]$ and $Q$ is smooth, we immediately get
\begin{align*}
\lim_{\epsilon \rightarrow 0} \int_{\R^d} \Lambda^{\phi^{\epsilon}(x), Q}(t,z) dz & = \lim_{\epsilon \rightarrow 0} (-1)^{ |\alpha|} \int_0^t q(s,\phi_s^{\epsilon}(x)) ds 
 = (-1)^{ |\alpha|} \int_0^t q(s,\phi_s(x)) ds 
 = \int_{\R^d} \Lambda^{\phi(x), Q}(t,z) dz 
\end{align*}

by dominated convergence.

Stability in $Q$ as a mapping $L^1(\R^d; L^{\infty}([0,T]; \R^d)) \rightarrow L^m(\Omega)$ follows from the linearity of the mapping $Q \rightarrow \int_{\R^d} \Lambda^{\phi(x), Q}(t,z) dz$ as well as the bounds from Lemma \ref{exponentialBoundsLambda}.

\end{remark}

\subsection{Convergence in H\"{o}lder spaces}

With the notation of Proposition \ref{summaryProposition} we shall need a result to ensure convergence of $\nu( f(\phi^n_{\cdot}))$ is uniform on a set of full measure.

\begin{proposition} \label{HolderConvergence}
Let $\gamma \in (0, H)$, $f \in C^{1}_b(\R^d;\R^d)$ and $\nu$ be a finite signed measure on $\R^d$. Then there exists 
a set $\Omega_{\gamma, \nu}$ of full measure such that 
$$
\lim_{n \rightarrow \infty} \nu(f (\phi^n_{\cdot}( \omega)) = \nu( f (\phi_{\cdot}(\omega)))
$$ 
in $C^{\gamma}([0,T];\R^d)$ for all $\omega \in \Omega_{\gamma, \nu}$.

\end{proposition}

\begin{proof}
We begin by showing that $\nu( f (\phi_t^n)) \rightarrow \nu( f (\phi_t))$ in $L^2(\Omega)$ for every $t$. To see this, consider
\begin{align*}
E[ \left| \nu(f( \phi_t^n)) - \nu( f( \phi_t)) \right|^2 ] & = E[ \left| \nu(f(\phi_t^n) - f(\phi_t)) \right|^2 ] \\
& \leq |\nu|(\R^d) \|\nabla  f\|_{\infty} \int_{\R^d} E[ |\phi_t^n(x) - \phi_t(x)|^2 ] d\nu(x)  \rightarrow 0
\end{align*}
as $n \rightarrow \infty$ by dominated convergence, which proves the first claim.
  
Next we find a set universal in $t$ for which we have pointwise in $\omega$ convergence. Denote by $\{ q_j\}_{j=1}^{\infty}$ an enumeration of $[0,T] \cap \mathbb{Q}$. We may extract a subsequence $\{\nu( f( \phi^{n(k,1)}_{q_1}))\}_{k \geq 1} \subset \{ \nu(f(\phi^{n}_{q_1}))\}_{n \geq 1}$ such that 
$$
\lim_{k \rightarrow \infty} \nu(f(\phi^{n(k,1)}_{q_1}(\omega)) )= \nu(f(\phi_{q_1}(\omega)))
$$
for $\omega \in \Omega_1$ with full measure. Furthermore, we define inductively a subsequence $\{ \nu(f(\phi^{n(k,j+1)}_{q_{j+1}}))\}_{k \geq 1} \subset \{\nu(f(\phi^{n(k,j)}_{q_{j+1}}))\}_{k \geq 1}$ such that 
$$
\lim_{k \rightarrow \infty} \nu(f(\phi^{n(k,j+1)}_{q_{j+1}}(\omega))) = \nu(f(\phi_{q_{j+1}}(\omega)))
$$
for $\omega \in \Omega_{j+1}$ with full measure. Let $\Omega_0 = \cap_{j=1}^{\infty} \Omega_j$, so that we have
$$
\lim_{j \rightarrow \infty} \nu(f(\phi^{n(j,j)}_q(\omega))) = \nu(f(\phi_{q}(\omega)))
$$
for all $\omega \in \Omega_0$ and $q$ rational.

Now, we construct a set where $\{ \nu(f(\phi_{\cdot}^n))\}_{n \geq 1}$ is relatively compact in $C([0,T];\R^d)$. Let $\epsilon > 0$ be such that $\gamma < H-\epsilon$ and choose a subset $\Omega_{H-\epsilon}$ with full measure such $\phi_{\cdot}$ satisfies (\ref{flowSDE}) and for every $\omega \in \Omega_{H-\epsilon}$ we have
$$
(t \mapsto B_t(\omega)) \in C^{H-\epsilon}([0,T];\R^d) .
$$
Note that $\nu(f(\phi_{\cdot}))$ is continuous on this set.

From (\ref{flowSDE}) we see  
\begin{align*}
  | \nu(f(\phi^n_t(\omega))) - \nu(f(\phi^n_s(\omega)))| &  \leq | \nu \left(f( \phi^n_t(\omega) )- f(\phi^n_s(\omega)) \right)| \\
  & \leq |\nu|(\R^d) \|\nabla f\|_{\infty} \left( \|b_n\|_{\infty}   |t-s| + |B_{st}(\omega)| \right) \\
& \leq |\nu|(\R^d)  \|\nabla f\|_{\infty} \left( \|b_n\|_{\infty} |t-s| + \|B(\omega)\|_{H-\epsilon} |t-s|^{H-\epsilon} \right)
\end{align*}
so that the uniform boundedness of $b_n$ implies that $\{ \nu(f(\phi^n_{\cdot}(\omega)))\}_{n \geq 1}$ is equicontinuous. Moreover, the sequence is bounded in $C([0,T];\R^d)$ and from Arzela-Ascoli's theorem there exists a converging subsequence
$\{\nu(f(\phi^{j(k,\omega)}_{\cdot}(\omega)))\}_{k \geq 1} \subset \{\nu(f(\phi^{n(j,j)}_{\cdot}(\omega)))\}_{j \geq 1}$. For $\omega \in \Omega_0 \cap \Omega_{H - \epsilon}$ - which has full measure - we see that the limit coincides with $\nu(f(\phi_{\cdot}(\omega)))$. Applying the above reasoning to any subsequence of $\{ \nu(f(\phi^n_{\cdot}(\omega )))\}_{n \geq 1}$ we get a further subsequence that converges to $ \nu(f(\phi_{\cdot}( \omega)))$ in $C([0,T];\R^d)$. Since $C([0,T];\R^d)$ is a Banach space this implies that the \emph{full} sequence converges. By interpolation of H\"{o}lder spaces we see that the claim is true if we let $\Omega_{\gamma, \nu} := \Omega_0 \cap \Omega_{H - \epsilon}$.
\end{proof}

\section{Continuity Equation}


In this section we want to study the rough linear continuity equation
\begin{equation} \label{roughCE}
\partial_t \mu_t + \textrm{div}( b \mu_t)  + \textrm{div}( \mu_t d\X_t) = 0 
\end{equation}
with given initial condition $\mu_0$.

\begin{definition}
Let $\mu_0$ be a finite signed measure on $\R^d$. A measure valued function $\mu : [0,T] \rightarrow \mathcal{M}(\R^d)$ is called a measure solution to (\ref{roughCE}) if 
$$
\mu_t + \int_0^t \textrm{div}(b(r, \cdot) \mu_r ) dr + \int_0^t \textrm{div}( \mu_r d\X_r) = \mu_0
$$
holds weakly in $\mathcal{M}(\R^d)$ meaning  for every $\eta \in C^{\infty}_c(\R^d)$ we have $\mu_{\cdot}( \nabla \eta ) \in \mathscr{D}_{\X}^{p \gamma}$ and
$$
\mu_t(\eta)  = \mu_0(\eta) + \int_0^t \mu_r(b(r,\cdot) \nabla\eta ) dr + \int_0^t \mu_r( \nabla \eta)  d\X_r .
$$
\end{definition}

\emph{If} we know that there exists a solution to
$$
\phi_t(x) = x + \int_0^t b(r,\phi_r(x))dr + X_t, 
$$
then for any test function $\eta \in C^{\infty}_c(\R^d)$ we have from Lemma \ref{ItoFormula}
$$
\eta(\phi_t(x)) = \eta(x) + \int_0^t \nabla\eta(\phi_r(x)) b(r,\phi_r(x)) dr + \int_0^t \nabla \eta(\phi_r(x)) d\X_r.
$$

We integrate the equation w.r.t. $\mu_0$ to see that $\mu_t := (\phi_t)_{\sharp} \mu_0$ solves (\ref{roughCE}) if we can use integration by parts for the rough path integral, namely
$$
\int_{\R^d} \int_0^t \nabla \eta(\phi_r(x)) d\X_r d\mu_0(x) = \int_0^t \mu_0(\nabla\eta(\phi_r)) d\X_r.
$$
Suppose now that $b \in C^1_b([0,T] \times \R^d;\R^d)$ and $\X \in \mathscr{C}^{\gamma}_g$. Let $X^{\epsilon} \in C^1([0,T] ;\R^d)$ be such that $ \X^{\epsilon} \rightarrow \X $ in $\mathscr{C}^{\gamma}$. Using Section \ref{ElementsOfControlledRoughPaths} we get
\begin{align*}
  \int_0^t \mu_0(\nabla\eta(\phi_r)) d\X_r   = \lim_{\epsilon \rightarrow 0} \int_0^t \mu_0(\nabla\eta(\phi^{\epsilon}_r)) \dot{X}^{\epsilon}_r dr  & = \lim_{\epsilon \rightarrow 0} \int_{\R^d} \int_0^t \nabla\eta(\phi^{\epsilon}_r(x)) \dot{X}^{\epsilon}_r dr  d\mu_0(x) \\ 
   & =  \int_{\R^d} \int_0^t \nabla\eta(\phi_r(x)) d\X_r  d\mu_0(x) . 
  \end{align*}


We summarize the above in a lemma.

\begin{lemma}
Suppose $b \in  C^1_b([0,T] \times \R^d;\R^d)$ and $\X \in \mathscr{C}^{\gamma}_g$. Then there exists a solution to (\ref{roughCE}) and the solution is given by $\mu_t := (\phi_t)_{\sharp} \mu_0$.
\end{lemma}

Given the previous sections the reader will not be surprised that we can extend this to when the drift is discontinuous provided we choose the rough path to be the lift of a fractional Brownian motion with low Hurst index.

\begin{lemma}
  Assume $ H <  \frac{1}{2(2d + 1)}$,  $b \in L^{\infty}([0,T] \times \R^d; \R^d) \cap L^1(\R^d; L^{\infty}([0,T]; \R^d)) $ and $\mu_0$ a finite signed measure on $\R^d$. There exists a subset $\Omega^* \subset \Omega$ with full measure such that for every $\omega \in \Omega^*$ we have
  \begin{itemize}
  \item
    The fractional Brownian motion lifts to a geometric rough path $\B(\omega) \in \mathscr{C}^{\gamma}_g$, $\gamma < H$.
  \item
    There exists a solution $\mu_{\cdot}(\omega)$ to
    $$
\mu_t(\omega) + \int_0^t \textrm{div}(b(r, \cdot) \mu_r(\omega) ) dr + \int_0^t \textrm{div}( \mu_r(\omega) d\B_r(\omega)) = \mu_0 . 
    $$
\end{itemize}

\end{lemma}

\begin{proof}
Denote by $\Omega_{\B}$ the set of $\omega \in \Omega$ such that $B(\omega)$ lifts to a rough path, $\B(\omega) \in \mathscr{C}^{\gamma}_g$.

Let $\eta \in C^{\infty}_c(\R^d)$.
Consider the approximation from Section \ref{fBmSDEChapter}, i.e. we have $\Omega_{\gamma,\eta,\mu_0}$ such that $\lim_{n \rightarrow \infty} \mu_0( \nabla \eta(\phi_{\cdot}^n(\omega))) = \mu_0( \nabla \eta(\phi_{\cdot}(\omega)))$ in $C^{\gamma}([0,T];\R^d)$. From Propositions \ref{summaryProposition} and \ref{HolderConvergence} we get that
$$
\int_0^{\cdot} \mu_0( \nabla  \eta(\phi_r^n(\omega))) d\B_r(\omega) \rightarrow \int_0^{\cdot} \mu_0( \nabla  \eta(\phi_r(\omega))) d\B_r(\omega)
$$
on $\Omega_{\gamma,\eta, \mu_0} \cap \Omega_{\B}$.

 For every $n$ we have that $\mu_t^n := (\phi_t^n)_{\sharp} \mu_0$ satisfies
 $$
\mu^n_t(\eta) = \mu_0(\eta) + \int_0^t \mu^n_r( b_n(r, \cdot) \nabla \eta) dr + \int_0^t \mu_r^n(\nabla\eta) d\B_r
$$
on $\Omega_{\B}$. Denote by $\Omega_{\eta, \mu_0}$ the set of $\omega \in \Omega$ such that $\mu_t^n(\eta) \rightarrow \mu_t(\eta)$, so that we must have that all the above terms converges on $\Omega_{\eta, \mu_0} \cap \Omega_{\gamma,\eta, \mu_0} \cap \Omega_{\B}$, to 
 $$
\mu_t(\eta) = \mu_0(\eta) + \int_0^t \mu_r( b(r, \cdot) \nabla \eta) dr + \int_0^t \mu_r(\nabla\eta) d\B_r.
$$
Let now
$$
\Omega^* := \Omega_{\B} \cap \bigcap_{k \geq 1} \Omega_{\eta^k, \mu_0} \cap \Omega_{\gamma,\eta^k, \mu_0}
$$
where $\{ \eta^k \}_{k \geq 1} \subset C_c^{\infty}(\R^d)$ is dense in $C^{\infty}_c(\R^d)$ equipped with the ususal test function topology. Then $\Omega^*$ is the desired set.
\end{proof}

\section{Transport Equation}


In this section we want to study (\ref{TEWeakIntegralForm}). Morally, the solution to this equation should be given by
$$
u(t,x) = u_0(\phi_t^{-1}(x)) \exp \{ - \int_0^t c(s, \phi_r(y))dr |_{y = \phi_t^{-1}(x)} \} .
$$
When $c$ is a distribution this expression does not make sense. Using chapter \ref{flowLocalTime} we can however \emph{define} the solution to be
$$
u(t,x) = u_0(\phi_t^{-1}(x)) \exp \{ (-1)^{|\alpha| + 1} \int_{\R^d} \Lambda^{\phi(y), C}_{\alpha}(t,z)dz |_{y = \phi_t^{-1}(x)} \} .
$$
where $C = D^{\alpha} c$ and $C$ is a function. 

Another question is in what way does this function defined above satisfy (\ref{TEWeakIntegralForm}). To answer this we should look for a spatially weak formulation of the equation, namely for every $\eta \in C^{\infty}_c(\R^d)$ the function should satisfy
$$
\langle u(t), \eta \rangle + \int_0^t \langle b(r) \nabla u(r), \eta \rangle  + \int_0^t \langle u(r) c(r), \eta \rangle  + \int_0^t \langle \nabla u(r),  \eta \rangle dB_r= \langle u_0, \eta \rangle   .
$$


In order to make sense of the stochastic integral term we need to guarantee that $\langle \nabla u(r),  \eta \rangle$ is a path controlled by $B$ as described in Section \ref{ElementsOfControlledRoughPaths}. Using integration by parts we get
\begin{align*}
\langle  \nabla u(r),  \eta \rangle &  = - \int_{\R^d} u_0(\phi_r^{-1}(x)) \exp \{ (-1)^{|\alpha| + 1} \int_{\R^d} \Lambda^{\phi(y), C}_{\alpha}(r,z)dz |_{y = \phi_r^{-1}(x)} \} \nabla \eta(x) dx \\
& =  - \int_{\R^d} u_0(y) \exp \{ (-1)^{|\alpha| + 1} \int_{\R^d} \Lambda^{\phi(y), C}_{\alpha}(r,z)dz  \} \nabla \eta( \phi_r(y) ) |\nabla \phi_r(y)| dy
\end{align*}
where we have used the change of variables $\phi_r(y) = x$. It is clear from Section \ref{controllingODEs} that $\nabla \eta(\phi_{\cdot}(y))$ can be regarded as a controlled path. However, the terms 
$$
\int_{\R^d} \Lambda^{\phi(y), C}_{\alpha}(\cdot,z)dz \hspace{.5cm} \textrm{ and } \hspace{.5cm}  |\nabla \phi_{\cdot}(y)| = \exp\{ - \sum_{j=1}^d \int_{\R^d} \Lambda^{\phi(y), b_j}_{e_j}(\cdot,z)dz \}
$$ 
are not expected to be more than $1 - H(2+d)$ regular in time (at least at the current level of knowledge) so we can not invoke Lemma \ref{scalarMultiplication} and it is not clear how to define the product as a controlled path. In fact this seems to require that also e.g. $|\nabla \phi_{\cdot}(y)|$ is controlled by $B$ and we do not yet know how do this construction.

In its full generality we still cannot show that $u$ defined as above solves the equation, but we provide some examples ($d=1$, $ div(b)$ bounded, $c = div(b)$ and time-homogenuous drift) where we can.

First, let us study the equation when the coefficients and the noise are regular.

\subsection{Regular Case}

Assume for a moment that the drift $b \in L^{\infty} ([0,T] ; C^1_b(\R^d))$  and we want to study the rough linear transport equation 
\begin{equation} \label{roughTE}
\partial_t u + b \nabla u  + c u +  \nabla u d \textbf{X}_t = 0 
\end{equation}
with given initial condition $u|_{t=0} = u_0$.

If we assume that $\X$ is the geometric lift of a smooth path $X \in C^1$, we may read (\ref{classicalTE}) in a classical way:
\begin{equation} \label{classicalTE}
\partial_t u(t,x) + b(t,x) \cdot \nabla u(t,x)  +c(t,x) u(t,x) + \nabla u(t,x)  \cdot \dot{X}_t = 0 
\end{equation}
with initial condition $u(0,x) = u_0(x)$. 
To solve this equation, let us define 
$$
u(t,x) := u_0(\phi_t^{-1}(x))\exp\{ -\int_0^t c(r,\phi_r(y))dr |_{y = \phi_t^{-1}(x)} \}
$$ 
where $\phi_t(x)$ is the solution to (\ref{roughODE}). Immediately,  $u(t,\phi_t(x)) = u_0(x)\exp\{ -\int_0^t c(r,\phi_r(x))dr \}$ and so
\begin{align*}
-c(t,\phi_t(x)) u_0(x) &\exp\{  -\int_0^t c(r,\phi_r(x))dr \}   = \frac{d}{dt} u(t,\phi_t(x)) \\
 & = \partial_t u(t,\phi_t(x)) + \nabla u(t,\phi_t(x))  \cdot \dot{\phi}_t(x) \\
 & = \partial_t u(t,\phi_t(x)) + \nabla u(t,\phi_t(x))  \cdot b(t,\phi_t(x)) + \nabla u(t,\phi_t(x)) \cdot \dot{X}_t  .
\end{align*}
Making a change of variables we see that $u(t,x) $ is indeed a solution of (\ref{classicalTE}).

Integrating the above w.r.t. $t$ and approximating a rough path $\X$ by smooth paths and taking the limit, it is reasonable that we should get
\begin{equation*} 
u(t,x) + \int_0^t b(r,x)  \cdot \nabla u(r,x) dr  + \int_0^t c(r,x) u(r,x)dr +  \int_0^t \nabla u(r,x)  d\X_r = u_0(x) 
\end{equation*}
provided the solution is such that $\nabla u(\cdot,x)$ is controlled by $\X$. Unfortunately, to guarantee that $\nabla u(t,x)$ is a controlled path we need higher order differentiability of the solution than the regularization of the fractional noise can provide. To circumvent this we use a spatially weak notion of solution.

\begin{definition} \label{Def:WeakControlledSolution}
Let 
$u_0, c: [0,T] \times \R^d \rightarrow \R$ and $b: [0,T] \times \R^d \rightarrow \R^d$ be given locally integrable functions.
Let $u : [0,T] \times \R^d \rightarrow \R$ be such that for all $t \in [0,T]$ we have $u(t, \cdot) \in W^{1,2}(U)$ for all open and bounded $U \subset \R^d$. We call $u$ a weak controlled solution to (\ref{roughTE}) if for all $\eta \in C^{\infty}_c(\R^d)$ the path $\int_{\R^d}\nabla u(\cdot, x) \eta(x) dx$ is controlled by $\X$ and the following equality holds
\begin{align} \label{integralTE}
\int_{\R^d} u(t,x) \eta(x) dx & + \int_0^t \int_{\R^d} \nabla u(r,x)  \cdot b(r,x) \eta(x) dx dr +  \int_0^t \int_{\R^d}  u(r,x) c(r,x) \eta(x) dx dr \\
& + \int_0^t \int_{\R^d} \nabla u(s, x) \eta(x) dx d \X_r  = \int_{\R^d} u_0(x) \eta(x)dx \notag .
\end{align}

\end{definition}

Existence of such a solution when the drift is nice is relatively straightforward. The proof is a consequence of the discussion in Section \ref{stabilityODEs} together with the above computations.

\begin{lemma}
Assume $b,c \in L^{\infty} ([0,T] ; C^1_b(\R^d))$,  and $\X \in \mathscr{C}_g^{\gamma}$. Then there exists a weak solution to (\ref{roughTE}).

\end{lemma}

\begin{proof}
Consider a smooth approximation $\X^{\epsilon}$ of $\X$ and let 
$$
u^{\epsilon}(t,x) := u_0(\phi_t^{\epsilon,-1}(x))\exp\{ -\int_0^t c(r,\phi^{\epsilon}_r(y))dr|_{y = \phi_t^{\epsilon, -1}(x)} \} 
$$ 
so that $u^{\epsilon}$ satisfies
\begin{align*} 
\int_{\R^d} u^{\epsilon}(t,x) \eta(x) dx & + \int_0^t \int_{\R^d} \nabla u^{\epsilon}(r,x) \cdot b(r,x) \eta(x) dx dr + \int_0^t \int_{\R^d} c(r,x) u^{\epsilon}(r,x) \eta(x) dx dr \\
&  +   \int_0^t \int_{\R^d} \nabla u^{\epsilon}(r,x) \eta(x) dx \dot{X}^{\epsilon}_rdr   = \int_{\R^d} u_0(x) \eta(x)dx \notag .
\end{align*}

Consider now $\int_{\R} \nabla u^{\epsilon}(r,x) \eta(x) dx$ as above. Using integration by parts we get 
\begin{align*}
\int_{\R^d} \nabla u^{\epsilon}(r,x) \eta(x) dx & = - \int_{\R^d} u_0(\phi_r^{\epsilon,-1}(x)) \exp\{ -\int_0^r c(r,\phi^{\epsilon}_s(y))ds|_{y = \phi_r^{\epsilon, -1}(x)} \} \nabla \eta(x) dx \\
& = - \int_{\R^d} u_0(y) \exp\{ -\int_0^r c(s,\phi^{\epsilon}_s(y))ds \} |\nabla \phi^{\epsilon}_r(y)| \eta(\phi_r^{\epsilon}(y)) dy 
\end{align*}
where we have used a change of variable $y = \phi_r^{\epsilon,-1}(x)$ in the last equality. From Liouville's formula we get $|\nabla \phi_r^{\epsilon}(y)| = \exp \{ \int_0^r div(b)(s,\phi^{\epsilon}_s(y))ds \}$

From Section \ref{controllingODEs}, if we can show that $\exp\{ \int_0^rdiv(b)(s,\phi^{\epsilon}_s(y))-  c(s,\phi^{\epsilon}_s(y))ds \}  \eta(\phi_r^{\epsilon}(y))$ converges in $\mathscr{D}_{\X}^{p \gamma}$ to $\exp\{ \int_0^rdiv(b)(s,\phi_s(y))-  c(s,\phi_s(y))ds \}  \eta(\phi_r(y))$, then it follows immediately that 
$$
\int_0^t \int_{\R^d} \nabla u^{\epsilon}(r,x) \eta(x) dx \dot{X}^{\epsilon}_rdr \rightarrow \int_0^t \int_{\R^d} \nabla u(r,x) \eta(x) dx d\X_r .
$$

To this end, we notice that from Lemma \ref{scalarMultiplication} it is enough to prove that $\int_0^{\cdot}div(b)(s,\phi^{\epsilon}_s(y))-  c(s,\phi^{\epsilon}_s(y))ds$ converges in $C^{\beta}$ to $\int_0^{\cdot}div(b)(s,\phi_s(y))-  c(s,\phi_s(y))ds$. From H\"{o}lder's inequality we get
\begin{align*}
|\int_r^t div(b)(s,\phi^{\epsilon}_s(y))-  c(s,\phi^{\epsilon}_s(y)) - div(b)(s,\phi_s(y)) +   c(s,\phi_s(y))ds | \\
 \leq |t-r|^{\beta} \| div(b)(\phi^{\epsilon}(y))-  c(\phi^{\epsilon}(y)) - div(b)(\phi(y)) +   c(\phi(y)) \|_{L^{1/\beta}([0,T])} .
\end{align*}
The result follows by dominated convergence and continuity (of $c$ and $div(b)$) as long as we choose $\beta = p \gamma < 1$.

Convergence of the remaining terms follows by similar considerations.
\end{proof}

\subsection{Singular case}

Motivated by the previous section we define our solution via the flow transformation.

\begin{definition} \label{Def:LocalTimeSolution}
Let $b: [0,T] \times \R^d \rightarrow \R^d$ be a given function and $c : [0,T] \rightarrow \mathcal{D}'(\R^d)$ be a distribution such that there exists functions $C_j : [0,T] \times \R^d \rightarrow \R^d$ for $j=1, \dots, J$ and multiindices $\alpha_1, \dots, \alpha_J$ satisfying $c(t) = \sum_{j=1}^J D^{\alpha_j}C_j(t)$ where $D^{\alpha_j}$ denotes spatial differentiation in the weak sense. We call the function 
$$
u(t,x) = u_0(\phi_t^{-1}(x)) \exp \{ \sum_{j=1}^J (-1)^{|\alpha_j| + 1} \int_{\R^d} \Lambda^{\phi(y), C_j}_{\alpha_j}(t,z)dz |_{y = \phi_t^{-1}(x)} \} .
$$
a local time solution of (\ref{roughTE}) with initial condition $u_0$ provided all the terms exists as in Section \ref{localTimeOfTheFlow}

\end{definition}

We go on to prove existence of such a solution for almost all paths of the fBm.

\begin{theorem}
Assume we have
\begin{itemize}
\item
$b \in L^{\infty}([0,T] \times \R^d; \R^d) \cap L^1(\R^d; L^{\infty}([0,T]; \R^d))$

\item
There exists smooth functions $C_j^{k}$ such that $\int_{\R^d} \sup_{t \in [0,T]} |C_j(t,y) - C_j^k (t,y)| dy \rightarrow 0$ as $k \rightarrow \infty$ for all $j$,

\item
$u_0$ is continuous,

\item
$|\alpha_j | \leq 1$

\item
$B$ is a fBm with Hurst parameter $H < \frac{1}{d + 2}$.

\end{itemize}

Then there exists a set of full measure, $\Omega_0$ such that for every $\omega \in \Omega_0$ there exists a local time solution of (\ref{roughTE}).

\end{theorem}

\begin{proof}
The proof is done by approximation of $b$ and then $c$ as in the above assumptions. For notational simplicity we assume $J=1$. Let $\Omega_{\gamma, \delta_x}$ be as in Proposition \ref{HolderConvergence} where $\delta_x$ is the Dirac centered at $x$, so that we have $u_0(\phi_n(t,x,\omega)^{-1}) \rightarrow u_0(\phi(t,x,\omega)^{-1})$. For a fixed $k$ we have
\begin{align*}
\lim_{n \rightarrow \infty} (-1)^{|\alpha|} \int_{\R^d} \Lambda_{\alpha}^{C^k,\phi_n(x)}(t,y)dy & = \lim_{n \rightarrow \infty} \int_0^t c^k(s,\phi_n(s,x)) ds \\
 & =  \int_0^t c^k(s,\phi_s(x)) ds = (-1)^{|\alpha|} \int_{\R^d} \Lambda_{\alpha}^{C^k,\phi(x)}(t,y) dy
\end{align*}
on a set $\Omega_k$ of full measure. Finally, we notice that 
\begin{align*}
E[ \left( \int_{\R^d} \Lambda_{\alpha}^{C^k,\phi(x)}(t,y) dy - \int_{\R^d} \Lambda_{\alpha}^{C,\phi(x)}(t,y) dy \right)^m ] & = E[ \left( \int_{\R^d} \Lambda_{\alpha}^{C^k- C,\phi(x)}(t,y) dy  \right)^m ] \\
 & \leq C_m \int_{\R^d} \| C^k(\cdot,y) - C(\cdot,y) \|_{\infty} dy \rightarrow 0
\end{align*}
by assumption, and thus there exists a subsequence and a set of full measure, $\tilde{\Omega}$ such that we have $ \lim_{k \rightarrow \infty} \int_{\R^d} \Lambda_{\alpha}^{C^k,\phi(x)}(t,y) dy = \int_{\R^d} \Lambda_{\alpha}^{C,\phi(x)}(t,y) dy$ on $\tilde{\Omega}$. The result follows when we choose $\Omega_0 =\Omega_{\gamma, \delta_x} \cap \tilde{\Omega} \cap  \cap_{k \geq 1} \Omega_k$.
\end{proof}

\begin{example}[The continuity equation revisited]
Let $c = div(b) = \sum_{j=1}^d \frac{ \partial b_j}{\partial x_j}$ where $b$ is as before. 
We get from the previous Theorem that the solution to
\begin{align*}
\int_{\R^d} u(t,x) \eta(x) dx & + \int_0^t \int_{\R^d} \nabla u(r,x) b(r,x) \eta(x) dx dr +  \int_0^t \int_{\R^d}  div(b)(r,x) u(r,x) \eta(x) dx dr \\
& + \int_0^t \int_{\R^d} \nabla u(s, x) \eta(x) dx d \B_r  = \int_{\R^d} u_0(x) \eta(x)dx \notag .
\end{align*}
is given by (actually by definition of the solution)
$$
u(t,x) = u_0(\phi_t^{-1}(x)) \exp \{  \sum_{j=1}^d  \int_{\R^d} \Lambda^{\phi(y), b_j}_{e_j}(t,z)dz |_{y = \phi_t^{-1}(x)} \}.
$$
Rewriting the above equation 
\begin{align*}
\int_{\R^d} u(t,x) \eta(x) dx & + \int_0^t \int_{\R^d} div (u(r,x) b(r,x)) \eta(x) dx dr   \\
& + \int_0^t \int_{\R^d} \nabla u(s, x) \eta(x) dx d \B_r  = \int_{\R^d} u_0(x) \eta(x)dx \notag 
\end{align*}
this should give us the same solution as the continuity equation if $\frac{d \mu_0}{dx} = u_0(x)$, i.e. $u_0$ is the Radon-Nikodym of the measure $\mu_0$ w.r.t. Lebesgue measure.

To see that this is indeed the case we consider again the approximation from Section \ref{fBmSDEChapter}. The solution of the continuity equation $\mu_t^n = (\phi_n(t,\cdot))_{\sharp} \mu_0$ so that for any $\eta \in C^{\infty}_c(\R^d)$ we have
\begin{align*}
\mu_t^n(\eta) & = \int_{\R^d} \eta(\phi_n(t,y)) u_0(y) dy  = \int_{R^d} \eta(x) u_0(\phi_n^{-1}(t,x)) \exp \{ -\int_0^t div(b_n(r,\phi_n(r,y)))dr|_{y= \phi_n^{-1}(t,x)} \} dx
\end{align*}
where we have used the change of variable $y = \phi_n^{-1}(t,x)$. As in the proof of the previous Theorem we can let $n \rightarrow \infty$ and find a set of full measure for which 
\begin{align*}
\mu_t(\eta) & =  \int_{R^d} \eta(x) u_0(\phi_t^{-1}(x)) \exp \{ \sum_{j=1}^d  \int_{\R^d} \Lambda^{\phi(y), b_j}_{e_j}(t,z)dz |_{y = \phi_t^{-1}(x)} \} dx  = \int_{R^d} \eta(\phi_t(y)) u_0(y) dy 
\end{align*}
since $|\nabla \phi_t(y)| = \exp \{ \sum_{j=1}^d  \int_{\R^d} \Lambda^{\phi(y), b_j}_{e_j}(t,z)dz  \}$. The latter expression is obviously a controlled path. We conclude that our definitions \ref{Def:WeakControlledSolution} and \ref{Def:LocalTimeSolution} coincides in this case, justifying definition \ref{Def:LocalTimeSolution} as more than just a limit object, but something that actually satisfies the equation in a reasonable sense.

\end{example}

\subsection{Local time solutions that are weak controlled solutions}

In this section we look at examples of $b$ and $c$ for which we can show that the local time in Definition \ref{Def:LocalTimeSolution} solutions are really solutions in the sense of Definition \ref{Def:WeakControlledSolution}.

We recall that we have to make sense of 
$$
\langle u(t),  \eta \rangle = - \int_{\R^d} u_0(y) \eta( \phi_t(y) ) |\nabla \phi_t(y)| \exp \{ \sum_{j=1}^J (-1)^{|\alpha_j| + 1} \int_{\R^d} \Lambda^{\phi(y), C_j}_{\alpha_j}(t,z)dz  \} dy 
$$
as a controlled rough path. 

When $c$ is a bounded function, the exponential term does not pose any problems: writing 
$$
\sum_{j=1}^J (-1)^{|\alpha_j| + 1} \int_{\R^d} \Lambda^{\phi(y), C_j}_{\alpha_j}(t,z)dz = \int_0^{t } c(r, \phi_r(y))dr 
$$ 
shows that this term Lipschitz in $t$, and so using Lemma \ref{scalarMultiplication} it is clear that this term can always be considered a controlled path. For this reason, we shall for the rest of this section assume for simplicity that $c= 0$. The extension to bounded $c$ is straightforward.

\subsubsection{One spatial dimension}

Consider the approximation $\phi_n(t,x)$ from section \ref{singularSDE}, i.e. we have $\phi_n(t,x) \rightarrow \phi_t(x)$ in $L^2(\Omega)$ and $\phi_n(t,\cdot) \rightarrow \phi_t$ weakly in $L^2(\Omega; W^{1,2}(U))$ (for simplicity we omit the subsequence). 

\begin{theorem} \label{Thm:OneDim}
Assume $b \in L^{\infty}( [0,T] \times \R) \cap L^1(\R^d; L^{\infty}([0,T]))$, $u_0 \in C^1(\R)$ such that $u_0' \in L^1(\R)$ and $H< \frac{1}{6}$. There exists a subset with full measure $\Omega^* \subset \Omega$, such that there exists a weak controlled solution to (\ref{roughTE}) w.r.t. every $B_{\cdot}(\omega)$ for $\omega \in \Omega^*$.

\end{theorem}

\begin{proof}
Define $u(t,x) := u_0(\phi_t^{-1}(x))$ and fix $\Omega_{\gamma}$ as in Proposition \ref{HolderConvergence}. We begin by showing that $\int_{\R} u(\cdot,x,\omega) \eta'(x) dx$ is controlled by $\B_{\cdot}(\omega)$ for every $\omega \in \Omega_{\gamma}$. It is enough to show
$$
\int_{\R} u(\cdot,x,\omega) \eta'(x) dx = - \int_{\R} u_0'(y) \eta(\phi_{\cdot}(y,\omega)) dy.
$$
To this end, note that for every $n$ we have
\begin{align*}
\int_{\R} u_n(t,x,\omega) \eta'(x) dx  = - \int_{\R} \nabla u_n(t,x,\omega) \eta(x) dx  & = - \int_{R} u_0' (\phi_n^{-1}(t,x,\omega))  \nabla\phi_n^{-1}(t,x,\omega) \eta(x) dx \\
& = - \int_{\R} u_0'(y) \eta(\phi_n(t,y,\omega)) dy.
\end{align*}
where we have used the change of variables $y = \phi_n^{-1}(t,x,\omega)$. Letting $n \rightarrow \infty$ we get $\eta(\phi_n(t,x,\omega)) \rightarrow \eta(\phi_t(x,\omega))$ and $u_n(t,x,\omega) \rightarrow u(t,x,\omega)$ from Remark \ref{inverseConvergence}. The desired equality holds from dominated convergence.

For every $n$ we have
\begin{align*} 
\int_{\R} u_n(t,x,\omega) \eta(x) dx  + \int_0^t \int_{\R} \nabla u_n(r,x,\omega) b_n(r,x) \eta(x) dx dr  - &  \int_0^t \int_{\R} u_n(t,x)  \eta( \phi_n(r,x,\omega)) dx  d \B_r (\omega)   \\
 & = \int_{\R} u_0(x) \eta(x)dx  
\end{align*}
where $\B_{st}(\omega)$ is the geometric lift of the fractional Brownian motion. We see that 
$$
\lim_{n \rightarrow \infty} \int_{\R} u_n(t,x,\omega) \eta(x) dx = \int_{\R} u(t,x,\omega) \eta(x) dx
$$
and 
$$
\lim_{n \rightarrow \infty} \int_0^t \int_{\R} u'_0(x) \eta( \phi_n(r,x,\omega))  dx d \B_r (\omega)  =  \int_0^t  \int_{\R} u'_0(x) \eta( \phi_r(x,\omega))  dx d \B_r (\omega) .
$$
where we have used Proposition \ref{summaryProposition} since $u_0' \in L^1(\R)$. Consequently, we must have that 
$$
\int_0^t \int_{\R} \nabla u_n(r,x,\omega) b_n(r,x) \eta(x) dx dr
$$ 
is converging. Now we get
\begin{align*}
\int_{\R} \nabla u_n(r,x,\omega) b_n(r,x) \eta(x) dx & = \int_{\R} \nabla u_n(r,x,\omega) (b_n(r,x) - b(r,x)) \eta(x) dx  + \int_{\R} \nabla u_n(r,x,\omega) b(r,x) \eta(x) dx .
\end{align*}

For the first term, from Remark \ref{inverseConvergence} we see that 
$$
\int_{\R} \nabla u_n(r,x,\omega) b(r,x) \eta(x) dx \rightarrow \int_{\R} \nabla u(r,x,\omega) b(r,x) \eta(x) dx,
$$
for all $r \in [0,T]$.

For the second term we take the expectation
\begin{align*}
E[|\int_{\R} \nabla u_n(r,x) (b_n(r,x) - b(r,x)) & \eta(x) dx|]  \leq  \int_{\R} E[ |\nabla u_n(r,x)| ] |b_n(r,x) - b(r,x)| |\eta(x)| dx \\
 \leq &  \sup_{y \in supp \eta} \left( E[ |\nabla u_n(r,y)|^2 ] \right)^{1/2} \int_{\R}  |b_n(r,x) - b(r,x)| |\eta(x)| dx \\
& \rightarrow 0 .
\end{align*}
which shows that there exists a subsequence 
$$
\int_0^t \int_{\R} \nabla u_{n_k}(r,x) b_{n_k}(r,x) \eta(x) dx dr \rightarrow \int_0^t \int_{\R} \nabla u(r,x) b(r,x) \eta(x) dx dr 
$$ 
as $k \rightarrow \infty$ on some set $\Omega_1$ which has full measure. The statement of the theorem is proved to be true if we let $\Omega^{*} = \Omega_{\gamma} \cap \Omega_1$.
\end{proof}

\subsubsection{Divergence of $b$ bounded}

When the divergence of $b$ is bounded, we can write
$$
|\nabla \phi_t(x) | = \exp \{ \int_0^t div b(r, \phi_r(x)) dr \} ,
$$
and so the mapping $t \mapsto |\nabla \phi_t(x) |$ is of bounded variation. Using Lemma \ref{scalarMultiplication} we can show the following result.

\begin{theorem}
Assume $b: [0,T] \times \R^d \rightarrow \R^d$ satisfies the assumptions of Proposition \ref{compactSequence} and has bounded divergence. Assume moreover that $u_0 :[0,T] \times \R^d \rightarrow \R$ is a bounded function. Then, if $u$ is a local time solution it is also a weak controlled solution. 

\end{theorem}

\begin{proof}
We need to show that $u(t,x) = u_0( \phi_t^{-1}(x))$ is a weak controlled solution. Using Lemma \ref{scalarMultiplication} and Lemma \ref{Lemma:FunctionComposition} it is clear that
$$
\langle u(t), \nabla \eta \rangle  = \int_{\R^d} u_0(y) \nabla \eta( \phi_t(y)) \exp \{ \int_0^t div b(r, \phi_r(y)) dr \} dy
$$
is controlled by $\B$. The proof follows the same lines as the proof of \ref{Thm:OneDim}, using Proposition \ref{compactSequence} to obtain strong $L^2(\Omega)$ convergence of $u_n(t,x)$ locally in $x$.
\end{proof}

\subsubsection{Time-homogenuous drift and smooth initial data}

When $b$ is time-homogenuous, we can write (see Remark \ref{inverseConvergence}) 
$$
\phi_t^{-1}(x) = x - \int_0^t b(\phi_r^{-1}(x)) dr - B_t .
$$
If now the initial condition is sufficiently regular, it is clear from Lemma \ref{Lemma:FunctionComposition} that $u_0( \phi_{\cdot}^{-1}(x))$ is controlled by $\B$.

\begin{theorem}
Assume $b: [0,T] \times \R^d \rightarrow \R^d$ satisfies the assumptions of Proposition \ref{compactSequence} and assume that $u_0 \in C_b^{ k} ( \R^d; \R) $ for some $k \geq  \lfloor \frac{1}{H} \rfloor$. Then, if $u$ is a local time solution it is also a weak controlled solution.  

\end{theorem}

\begin{proof}
Since $\mathscr{D}_{- \B}^{p (H-)} = \mathscr{D}_{\B}^{p (H-)}$, it is clear that
$$
\langle  u(t), \nabla \eta \rangle  = \int u_0( \phi_{t}^{-1}(x) ) \nabla \eta(x) dx
$$
is controlled by $\B$. The proof follows the same lines as the proof of \ref{Thm:OneDim}, using Proposition \ref{compactSequence} to obtain strong $L^2(\Omega)$ convergence of $u_n(t,x)$ locally in $x$.
\end{proof}

\newpage

\end{document}